\documentclass[11pt]{article}

\usepackage{amsmath,amsthm,amsfonts,amssymb}
\usepackage{graphicx}
\usepackage{mathrsfs,bm,wasysym}
\usepackage{comment}
\usepackage{graphicx}
\usepackage{centernot}
\usepackage{enumerate}
\usepackage{comment}
\usepackage{thmtools}
\usepackage{float}
\usepackage{caption}
\usepackage{subcaption}
\usepackage{textcomp}
\usepackage[normalem]{ulem}
\usepackage{hyperref}
\usepackage{titlesec}
\usepackage{cleveref}

\parskip \medskipamount
\declaretheoremstyle[
qed={//}
]{defstyle}

\numberwithin{equation}{section}

\theoremstyle{plain}
\newtheorem{theorem}{Theorem}[section]
\newtheorem{lemma}[theorem]{Lemma}
\newtheorem{proposition}[theorem]{Proposition}
\newtheorem{corollary}[theorem]{Corollary}

\newtheorem{maintheorem}{Theorem}

\theoremstyle{remark}

\newtheorem{example}[theorem]{Example}
\newtheorem{definition}[theorem]{Definition}

\newcommand{\pr}[1]{\mathbb{P}\left(#1\right)}
\newcommand{\prcond}[2]{\mathbb{P}\left(#1\;\middle\vert\;#2\right)}
\newcommand{\prstart}[2]{\mathbb{P}_{#2}\left(#1\right)}
\newcommand{\E}[1]{\mathbb{E}\left[#1\right]}
\newcommand{\econd}[2]{\mathbb{E}\left[#1\;\middle\vert\;#2\right]}
\newcommand{\estart}[2]{\mathbb{E}_{#2}\left[#1\right]}
\newcommand{\Var}[1]{\mathrm{Var}\left(#1\right)}
\newcommand{\varcond}[2]{\mathrm{Var}\left(#1\;\middle\vert\;#2\right)}
\newcommand{\Cov}[2]{\mathrm{Cov}\left(#1,#2\right)}
\newcommand{\covcond}[3]{\mathrm{Cov}\left(#1,#2\;\middle\vert\;#3\right)}

\newcommand{\BCap}[1]{\mathrm{BCap}\left(#1\right)}
\newcommand{\RWCap}[1]{\mathrm{Cap}\left(#1\right)}
\newcommand{\1}[1]{{\text{\Large $\mathfrak 1$}}_{#1}}
\newcommand{\til}[1]{\widetilde{#1}}
\newcommand{\what}[1]{\widehat{#1}}
\newcommand{\Unif}[1]{\mathrm{Unif}\!\left(#1\right)}

\newcommand{\Geomnonneg}[1]{\mathrm{Geom}_{\ge0}\!\left(#1\right)}

\newcommand{\T}{\mathcal{T}}
\newcommand{\Scal}{\mathcal{S}}
\newcommand{\Xtil}{\til{X}}
\newcommand{\A}{\mathcal{A}}
\newcommand{\B}{\mathcal{B}}
\newcommand{\Ell}{\mathcal{L}}
\newcommand{\Ecal}{\mathcal{E}}
\newcommand{\xiln}{\xi^{\ell}_n}
\newcommand{\xirn}{\xi^{r}_n}

\newcommand{\eqdist}{\stackrel{\mathrm{d}}{=}}

\newcommand{\Z}{\mathbb{Z}}
\newcommand{\Zpos}{\Z_{\ge1}}
\newcommand{\Znonneg}{\Z_{\ge0}}

\newcommand{\GW}{\mathrm{GW}}
\newcommand{\LW}{\mathrm{LW}}
\newcommand{\LH}{\mathrm{LH}}

\newcommand{\eps}{\varepsilon}

\begin{document}

\title{Intersections of branching random walks on $\Z^8$}
\date{}
\author{Zsuzsanna Baran
\thanks{University of Cambridge, Cambridge, UK. {zb251@cam.ac.uk}}}

\maketitle

\begin{abstract}
We consider random walks on $\Z^8$ indexed by the infinite invariant tree, which consists of an infinite spine and finite random trees attached to it on both sides.
We establish the precise order of the non-intersection probability between one walk indexed by one side of the tree, and an independent one indexed by both sides of an independent tree. This is analogous to the result by Lawler from the '90s for two independent simple random walks on $\Z^4$. We also prove a weak law of large numbers for the branching capacity of the range of a branching random walk.
\end{abstract}

\section{Introduction}\label{sec:intro}

In this work we study the intersections of two branching random walks on $\Z^8$. To put this problem into a broader context and provide some motivation, we start by briefly discussing the analogous question for two simple random walks.

Let $X$ and $\til{X}$ be two independent simple random walks on $\Z^d$, starting from $0$, and let us consider the number of intersections $\sum_{i=0}^{n}\sum_{j=1}^{\infty}\1{X_i=\til{X}_j}$ between the ranges $X[0,n]$ and $\til{X}(0,\infty)$. It is known that in $d>2$ dimensions the Green's function of a simple random walk decays as $g(x)\asymp\frac{1}{||x||^{d-2}+1}$ (where $a\asymp b$ means that $a$ and $b$ are of the same order). If we make the rough approximation of replacing $X[0,n]$ with the part of the range of $X$ falling into the ball $B\left(0,\sqrt{n}\right)$, then a quick calculation suggests that the expected number of intersections is
\begin{align*}
\approx\:\E{\sum_{i=0}^{\infty}\sum_{j=1}^{\infty}\1{X_i=\til{X}_j\in B(0,\sqrt{n})}} \asymp\sum_{x\in B(0,\sqrt{n})}\left(\frac{1}{||x||^{d-2}+1}\right)^2\:\asymp\:
\begin{cases}
n^{\frac{4-d}{2}} &\text{ for }d<4,\\
\log n &\text{ for }d=4,\\
1&\text{ for }d>4.
\end{cases}
\end{align*}
A more careful calculation confirms that the expected number of intersections is indeed of this order (see e.g.~\cite[Proposition 3.2.3]{intersections_of_RWs}). We see that for this question the `critical dimension' is 4, since the expected number of intersections grows polynomially for $d<4$, logarithmically for $d=4$ and is of constant order for $d>4$. It was conjectured that the non-intersection probability of two independent walks shows a similar behaviour, it decays polynomially in $n$ for $d<4$, as a power of $\log n$ for $d=4$, and is of constant order for $d>4$.

In 1991 Lawler~\cite[first edition]{intersections_of_RWs} found the exact exponent of the logarithm and the exact constant factor for the case of a one-sided and a two-sided walk in 4 dimensions.
He proved that for two independent walks we have $\pr{X[-n,n]\cap\til{X}(0,\infty)=\emptyset}=\left(1+o(1)\right)\frac{\pi^2}{8}\frac{1}{\log n}$. In 1992~\cite{slowly_rec_sets} he also extended the result for two one-sided walks and proved that $\pr{X[0,n]\cap\til{X}(0,\infty)=\emptyset}=\left(1+o(1)\right)\frac{c}{\sqrt{\log n}}$.

We now move on to the setting of two branching random walks. We start by defining the infinite invariant tree, which is a random infinite tree that will be indexing the branching random walks.

\begin{definition}[Infinite invariant tree]\label{def:infinite_tree}
Let $\mu$ be a distribution on $\Znonneg$ with mean 1 and variance $\sigma^2\in(0,\infty)$. We consider a random infinite spatial tree $\T$ as follows.
\begin{itemize}
\item The tree has an infinite sequence of \emph{spine vertices} $(u_i)_{i\ge0}$ where $u_0$ is called the \emph{root}, and $u_i$ is an offspring of $u_{i-1}$ for all $i\ge1$.
\item The root has $d^+_0\sim\mu$ offsprings on the right side of the spine, and no offsprings on the left side.
\item For $i\ge1$, the spine vertex $u_i$ has $d^+_i$ offsprings on the right side of the spine and $d^-_i$ offsprings on the left side of the spine, where $\pr{d^+_i=k,\:d^-_i=\ell}=\mu(k+\ell+1)$, and the pairs $(d^{+}_i,d^-_i)$ are independent for different $i$, and also independent of $d^+_0$.
\item Each non-spine vertex has $\sim\mu$ offsprings, independently of each other and the above $d^{\pm}_i$. (In other words, given the offsprings of the spine vertices on each side of the spine, we attach independent $\GW(\mu)$ trees to each of them.)
\end{itemize}
We index the vertices of $\T$ by $\Z$ in the depth-first order in the anticlockwise direction, assigning index $0$ to the root. (See \Cref{fig:T} for an illustration.) We sometimes refer to the vertices with nonnegative labels as the \emph{future} of the tree, and to vertices with negative labels as the \emph{past} of the tree. We write $\T_j$ to mean the vertex with index $j$, write $\T[a,b]$ for the vertices with indices in $\{a,a+1,...,b\}$, $\T(a,b]$ for vertices with indices in $\{a+1,...,b\}$, etc.
\end{definition}

\begin{figure}%[h]
\centering
\includegraphics[width=100mm]{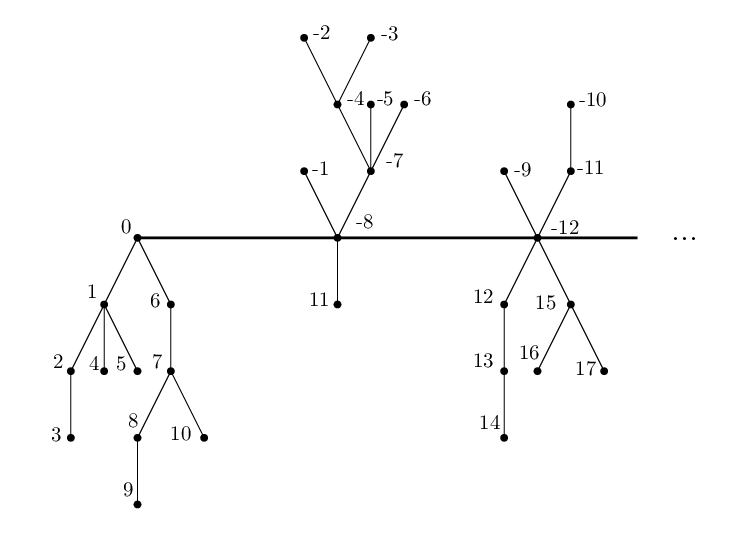}
\caption{Illustration of $\T$ and the indexing of its vertices.}
\label{fig:T}
\end{figure}

\begin{definition}[Branching random walk]\label{def:RW_on_tree}
Given a tree $T$, we define a random walk $(S_v)_{v\in T}$ indexed by the vertices of this tree as follows. Let $(D_v)_{v\in T}$ be iid random variables according to a given distribution, and for each vertex $v\in T$ let $S_v=\sum_{u}D_u$ where the sum ranges over the ancestors of $v$, not including the root, but including $v$ itself (unless it is the root).

In what follows, we assume that $(D_v)$ are distributed uniformly on the $2d$ neighbours of $0$ in $\Z^d$, and we consider a random walk $S$ on a tree as in \Cref{def:infinite_tree}. We call such a walk a \emph{branching random walk} (BRW).
\end{definition}

A closely related object, a random walk indexed by a critical Galton-Watson tree conditioned to survive, was considered by Kesten~\cite{Kesten1986} in 1986 as a simpler model to study alongside the incipient infinite cluster of a critical bond percolation. A branching random walk indexed by the past side of $\T$ from \Cref{def:infinite_tree} was introduced by Le Gall and Shen~\cite{range_of_tree-indexed_RW} in 2016, and later, in 2022 Bai and Wan~\cite{Cap_of_BRW} considered a branching random walk indexed by the whole of $\T$. This tree only differs from Kesten's one in the offspring distribution of the root vertex, but it has the useful property of shift-invariance (see \Cref{lem:shift_invariance}) which makes it more amenable to analyse.

In what follows, we let $\T$ be a tree as in \Cref{def:infinite_tree} and $\Scal$ the corresponding random walk as in \Cref{def:RW_on_tree}. To simplify notation, we will usually write $\T=(\T,\Scal)$ to denote both the tree and the walk on it.

We are interested in the non-intersection probability of two independent branching random walks. We can see that for $d>4$ the Green's function of a BRW decays as $G(x)\asymp\frac{1}{||x||^{d-4}+1}$ (see \Cref{lem:Greens_function}). It is known (see \Cref{lem:GW_tail_bound}) that the size of a critical Galton-Watson tree satisfies $\pr{|\GW|>k}\asymp\frac{1}{\sqrt{k}}$, therefore the overall size of the trees descending from the first $k$ spine vertices is typically $\asymp k^2$. Also, the walk along the spine typically travels distance $\asymp\sqrt{k}$ in the first $k$ steps.
Based on this, we may roughly approximate $\T[0,n]$ by $\T[0,\infty)\cap B\left(0,n^{\frac14}\right)$, and estimate the expected number of intersections $\sum_{i=0}^{n}\sum_{j=1}^{\infty}\1{\T_i=\til{\T}_j} $ between $\T[0,n]$ and $\til{\T}(0,\infty)$ as
\[\asymp\sum_{x\in B\left(0,n^{\frac14}\right)} \left(\frac{1}{||x||^{d-4}+1}\right)^2\quad\asymp\quad
\begin{cases}
n^{\frac{8-d}{4}} &\text{ for }d<8,\\
\log n &\text{ for }d=8,\\
1&\text{ for }d>8.
\end{cases}\]
This indicates that the critical dimension for this question is 8, and the expected number of intersections shows a similar behaviour to the case of two simple random walks. It also suggests that the non-intersection probability in the critical dimension might behave similarly to the case of two SRWs, decaying as a power of $\log n$.

For the case of a one-sided and a two-sided BRW we confirm this, as follows.

\begin{maintheorem}\label{thm:Ttiln_Tminusnn_prob}
Let $\mu$ be a distribution on $\Znonneg$ with mean 1 and variance $\sigma^2\in(0,\infty)$. Let $\T$ and $\til{\T}$ be two independent branching random walks on $\Z^8$ with offspring distribution $\mu$, as in \Cref{def:RW_on_tree}. Then for all sufficiently large $n$ we have
\[\frac{c}{\log n}\quad\le\quad\pr{\til{\T}(-\infty,0)\cap\T[-n,n]=\emptyset}\quad\le\quad\frac{C}{\log n},\]
where $c$ and $C$ are positive constants depending only on $\sigma^2$.

The same result also holds with $\til{\T}[-n,0)$ instead of $\til{\T}(-\infty,0)$.
\end{maintheorem}

If we also include the event that the future side of $\T$ does not hit 0, we can establish the precise asymptotics for the non-intersection probability as follows.

\begin{maintheorem}\label{thm:prob_nocap_nozero_n}
Let $\T$ and $\til{\T}$ be as in \Cref{thm:Ttiln_Tminusnn_prob}. Then we have
\[\pr{\til{\T}(-\infty,0)\cap\T[-n,n]=\emptyset,\:0\not\in\T(0,n]}\quad=\quad\left(1+o(1)\right)\frac{c_8}{\log n}\qquad\text{as }n\to\infty,\]
where $c_8$ is a positive constant depending only on $\sigma^2$.
\end{maintheorem}

Another question related to the intersections of two branching random walks is about the branching capacity of a branching random walk range.

Branching capacity was introduced by Zhu~\cite{BRWs_I} in 2017 as an analogue of the capacity $\RWCap{\cdot}$ defined for simple random walks. It is a notion of size based on the escape probabilities of a branching random walk, defined as follows.

\begin{definition}[Branching capacity]\label{def:BCap}
For a finite set $A\subseteq\Z^d$ we define its \emph{branching capacity} as
\[\BCap{A}:=\quad\sum_{x\in A}\pr{(x+\T(-\infty,0))\cap A=\emptyset},\]
where $\T$ is a branching random walk on $\Z^d$ as in \Cref{def:RW_on_tree}, and $x+\T(-\infty,0)$ denotes the Minkowski sum $\{x+y:\:y\in\T(-\infty,0)\}$.
\end{definition}

We obtain a weak law of large numbers for the branching capacity of a branching random walk range as follows.

\begin{maintheorem}\label{thm:BCap_asymp}
Let $\mu$ be a distribution on $\Znonneg$ with mean 1 and variance $\sigma^2\in(0,\infty)$. Let $\T$ be a branching random walk on $\Z^8$ with offspring distribution $\mu$, as in \Cref{def:RW_on_tree}, and consider the branching capacity with offspring distribution $\mu$ as in \Cref{def:BCap}. Then we have
\[\frac{\log n}{n}\:\BCap{\T[0,n]}\quad\to\quad c_8\qquad\text{in probability}\quad\text{as }n\to\infty,\]
where $c_8$ is the positive constant from \Cref{thm:prob_nocap_nozero_n}, depending only on $\sigma^2$.
\end{maintheorem}

An analogous result about the capacity of a simple random walk range in 4 dimensions was established by Asselah, Schapira and Sousi~\cite{Cap_of_SRW} in 2019. They proved that $\frac{\log n}{n}\RWCap{X[0,n]}\to\frac{\pi^2}{8}$ a.s..

So far we discussed previous results regarding two SRWs, but problems concerning the intersection of a SRW with a BRW have also been studied. From calculations analogous to the earlier ones we can see that the critical dimension for this case is 6. In 2022 Bai and Wan~\cite{Cap_of_BRW} studied the capacity of the range of a BRW and showed that in 6 dimensions it satisfies $\frac{\log n}{n}\RWCap{\T[0,n]}\to c$ in $L^2$ and in probability. They also showed that for an independent SRW and BRW on $\Z^6$ we have $\pr{\til{X}(0,\infty)\cap\T[-n,n]=\emptyset,\:0\not\in\T(0,n]}=\left(1+o(1)\right)\frac{c}{\log n}$ as $n\to\infty$. \footnote{For both of these results they also assumed that $\mu$ has a finite fifth moment.} In 2023 Schapira~\cite{BCap_of_RW_range} studied the branching capacity of the range of a SRW and showed that in 6 dimensions it satisfies $\frac{\log n}{n}\BCap{X[0,n]}\to c'$ in $L^2$ and a.s.. He also proved that for an independent SRW and BRW on $\Z^6$ we have $\pr{\til{\T}(-\infty,0)\cap X[-n,n]=\emptyset,\:0\not\in X(0,n]}=\left(1+o(1)\right)\frac{2\pi^3}{27\sigma^2}\frac{1}{\log n}$.

We focus on the critical dimension, but the above works also concern some of the other dimensions. Other recent developments for non-critical dimensions include~\cite{Cap_of_SRW_Z5}, \cite{Cap_of_BRW_low_dimensions} and~\cite{BCap_of_SRW_Z5}.

\subsection{Further directions}

Our overall goal is to gain a better understanding of the geometry of the range of a branching random walk in 8 dimensions. In order to do that we would like to understand how two BRW ranges intersect, which also helps understanding how one BRW range intersects itself.

In this paper we take a first step towards that by establishing the order of the non-intersection probability of a one-sided and a two-sided BRW started from the same point. In subsequent works we are also hoping to understand the non-intersection probability between two one-sided or two two-sided ranges started from the same point, and the related question of the intersection probability between two BRW ranges started from further apart. These could also be a helpful step towards understanding e.g.\ how the effective resistance on the graph induced by the range of a BRW grows, generalising the result of Shiraishi~\cite{RW_resistance} about the effective resistance on the range of a SRW on $\Z^4$.

\subsection{Overview}

In this section we explain the main ideas of the proofs, emphasising the new difficulties that arise compared to previous works. We follow the same outline as~\cite{Cap_of_BRW} and~\cite{BCap_of_RW_range}, which both build on the work of Lawler, but there are significant new ideas in how we can handle each step of the proof.

One of the main sources of difficulty is that a BRW does not have a Markov property like a SRW does. In previous results where at least one of the two objects concerned was a SRW, the Markov property of that walk was utilised, but for the case of two BRWs we are not able to use that. 

The following result will be a major step in the proof of Theorems~\ref{thm:Ttiln_Tminusnn_prob}, \ref{thm:prob_nocap_nozero_n} and~\ref{thm:BCap_asymp}.
\begin{maintheorem}\label{thm:prob_nocap_nozero}
Let $\T$ and $\til{\T}$ be as in \Cref{thm:Ttiln_Tminusnn_prob}, and let $\xiln$ and $\xirn$ be independent $\Geomnonneg{\frac1n}$ random variables. Then we have
\[\pr{\til{\T}(-\infty,0)\cap\T[-\xiln,\xirn]=\emptyset,\:0\not\in\T(0,\xirn]}\quad=\quad\left(1+o(1)\right)\frac{c_8}{\log n}\qquad\text{as }n\to\infty,\]
where $c_8$ is a positive constant depending only on $\sigma^2$.
\end{maintheorem}

Similarly to previous works, our starting point is a `magic formula' which states that
\begin{align}\label{eq:magic_formula_description}
\E{\1{\til{\T}(-\infty,0)\cap\T[-\xiln,\xirn]=\emptyset}\:\1{0\not\in\T(0,\xirn]}\:\#\left(\til{\T}[0,\infty)\cap\T[-\xiln,\xirn]\right)}\quad=\quad1.
\end{align}
We can see that the indicators correspond to the event in \Cref{thm:prob_nocap_nozero}, while the expectation of the counting term is $\asymp\log n$. Then our proof proceeds as follows.

\subsubsection*{Step 1: proving concentration of the number of intersections}
Firstly, we show that $\#\left(\til{\T}[0,\infty)\cap\T[-\xiln,\xirn]\right)$ is concentrated, namely we prove that
\begin{align*}
\E{\#\left(\til{\T}[0,\infty)\cap\T[-\xiln,\xirn]\right)}\sim\frac{\log n}{c_8},\qquad\Var{\#\left(\til{\T}[0,\infty)\cap\T[-\xiln,\xirn]\right)}\lesssim\log n,
\end{align*}
where $\sim$ and $\lesssim$ are defined as in \Cref{sec:other_notation}.

The main difficulty for the expectation is that it is not clear what the Green's function associated to the first given number of vertices in $\T$ behaves like, while the main difficulty for the variance is that we are not able to use the strong Markov property when considering the probability of $\T$ hitting two given vertices.

To estimate the expectation and the variance, we consider the so-called depth-first queue process encoding Galton-Watson trees, and use this to express the expected number of vertices in $\T[-\xiln,\xirn]$ at a given distance from the root, and the expected number of pairs in $\T[-\xiln,\xirn]$ at a given distance from the root and each other. Then using properties of the depth-first queue process we are able to reinterpret and simplify these expressions, and get sums that concern SRWs up to a geometric number of steps. (See \Cref{lem:GW_sums} for some of the most technical calculations.)

Then we use that considering the Green's function of a SRW up to $\Geomnonneg{\alpha}$ steps is similar to considering the Green's function for the whole walk, but only keeping the values at sites $x$ with $||x||\le\frac{C}{\sqrt{\alpha}}$. We also present the bounds we use to justify this approximation. (See \Cref{lem:g_alpha_x_bound}.)

\subsubsection*{Step 2: decorrelating the indicators and the number of intersections}

As a next step, we use the concentration of $\#\left(\til{\T}[0,\infty)\cap\T[-\xiln,\xirn]\right)$ to `pull it out' from the expectation in~\eqref{eq:magic_formula_description}. We can quickly get
\begin{align}\label{eq:prob_upper_bound}
\pr{{\til{\T}(-\infty,0)\cap\T[-\xiln,\xirn]=\emptyset},\:{0\not\in\T(0,\xirn]}}\quad\lesssim\quad\frac{1}{\log n},
\end{align}
but we also need a lower bound, and a more precise estimate for the upper bound.

We can see that the terms $\1{\til{\T}(-\infty,0)\cap\T[-\xiln,\xirn]=\emptyset}\1{0\not\in\T(0,\xirn]}$ and $\#\left(\til{\T}[0,\infty)\cap\T[-\xiln,\xirn]\right)$ have a dependence via $\T$ and via the spine of $\til{\T}$.

Inspired by~\cite[Section 3.3]{BCap_of_RW_range}, we split $\T$ into two parts, $\T(\le C)$ that is the part of the tree up to the first spine vertex where the walk hits a given distance $C$ from 0, and $\T(\ge C)$ that consists of the rest of the tree. We also split $\til{\T}$ into $\til{\T}(\le\til{C})$ and $\til{\T}(\ge\til{C})$ similarly.

Then we consider the event that (i) $\#\left(\T(\ge C)\cap\til{\T}(\le\til{C})\right)$ deviates a lot from its mean, (ii) $\T\left(\le\frac14C\right)$ does not intersect the negative side of $\til{\T}(\ge4\til{C})$, and (iii) the positive side of $\T\left(\le\frac14C\right)$ does not hit 0. We bound its probability by decorrelating the parts $\left(\T(\ge C),\til{\T}(\le\til{C})\right)$ from the parts $\left(\T\left(\le\frac14C\right),\til{\T}(\ge4\til{C})\right)$, and bounding the probability of (i) and (ii)$\cap$(iii) separately.

We also bound the probabilities of analogous events with other combinations of the parts $\T(\le C)$, $\T(\ge C)$, $\til{\T}(\le\til{C})$ and $\til{\T}(\ge\til{C})$.

This finishes the proof of \Cref{thm:prob_nocap_nozero}, and using the concentration of $\xiln$ and $\xirn$, we also get \Cref{thm:prob_nocap_nozero_n}. It also quickly implies asymptotics on $\E{\BCap{\T[0,n]}}$, and similarly to~\cite[Proposition 4.11]{Cap_of_BRW} we can bound $\Var{\BCap{\T[0,n]}}$ to finish the proof of \Cref{thm:BCap_asymp}.

\subsubsection*{Step 3: removing the $0\not\in\T\left(0,n\right]$ term}

For proving \Cref{thm:Ttiln_Tminusnn_prob} we have to remove the $0\not\in\T\left(0,n\right]$ condition from the probability in \Cref{thm:prob_nocap_nozero_n}.

In~\cite[Lemma 3.2]{BCap_of_RW_range}, where in place of $\T$ we have a SRW, this is done using the following idea. We can compare a SRW and a SRW that is conditioned not to hit 0 at positive times by considering a SRW and removing the part between time 0 and its last visit to 0. Since we do not have a strong Markov property for BRWs, we are not able to use this trick in our case.

Instead, we show that a BRW can be compared to a BRW that does not hit 0 at positive times and is weighted by the number of zeroes at negative times. More precisely, we show (in \Cref{lem:Ttilminus_Tminusnn_prob_as_E}) that
\begin{align*}
\pr{\til{\T}(-\infty,0)\cap\T[-n,n]=\emptyset}
\sim
\E{\1{\til{\T}(-\infty,0)\cap\T[-n,n]=\emptyset}\1{0\not\in\T(0,\infty)}\#\left(0\text{ in }\T(-\infty,0]\right)}+o\left(\frac{1}{\log n}\right).
\end{align*}
While we know that the expectation of the indicator term is of the desired order $\asymp\frac{1}{\log n}$, and the expectation of the counting term is $\asymp1$, it still requires considerable work to decorrelate these. We split the counting term based on the trees descending from different spine vertices, and use different estimates based on the value of the walk at the given spine vertex and the number of times this value was visited by earlier spine vertices. 

When making these estimates, we lose a constant factor, this is why in \Cref{thm:Ttiln_Tminusnn_prob} we can only estimate the probabilty up to a constant factor, and do not get exact asymptotics.

\subsection{Organisation}\label{sec:organisation}

In \Cref{sec:prelim} we introduce some notation and establish some preliminary results. In particular we state and prove the `magic formula', and discuss depth-first queue processes.

In \Cref{sec:E_and_Var_of_U} we prove the concentration of the number of intersections between the two trees. The calculations manipulating depth-first queue processes are found in \Cref{sec:GW_sums}.

In \Cref{sec:prob_nocap_nozero} we prove Theorems~\ref{thm:prob_nocap_nozero} and~\ref{thm:prob_nocap_nozero_n}, in \Cref{sec:Ttiln_Tminusnn_prob} we prove \Cref{thm:Ttiln_Tminusnn_prob}, and in \Cref{sec:BCap_asymp} we prove \Cref{thm:BCap_asymp}.

In \Cref{sec:mu} we briefly comment on the choice of $\mu$ in the results.

The statements and proofs of some auxiliary results are deferred to the Appendix.

\section{Some notation and preliminary results}\label{sec:prelim}

In this section we introduce some notation and present some preliminary results that will be used throughout the paper.

\subsection{Further notation regarding BRWs} \label{sec:further_notation_BRW}

As mentioned in \Cref{sec:intro}, we let $\T=(\T,\Scal)$ denote a random walk as in \Cref{def:RW_on_tree} indexed by a tree as in \Cref{def:infinite_tree}. We write $X$ for the simple random walk (SRW) along the spine of $\T$.

Unless specified otherwise, $\til{\T}=(\til{\T},\til{\Scal})$ denotes a branching random walk independent of $\T$, and we denote the associated quantities analogously as $\til{X}$, $\til{d}^{\pm}_i$, etc.

In some of the calculations we will also use SRWs and random walks indexed by a Galton-Watson tree. Unless specified otherwise, $T=(T,S)$ denotes a $\GW(\mu)$ tree and a random walk indexed by it, and $X$, $\til{X}$, and $\what{X}$ denote SRWs independent of each other and $T$.

We use the $L^2$-distance on $\Z^d$ and write $B(x,r)=\{y\in\Z^d:\:||y-x||\le r\}$, $\partial B(x,r)=\{y\in\Z^d:\:\lfloor||y-x||\rfloor=r\}$. In what follows we drop the $\lfloor\cdot\rfloor$ from the notation. We also write $\tau_r=\inf\{k\ge0:\:||X_k||=r\}$ and $\tau_x=\inf\{k\ge0:\:X_k=x\}$.

We write $\lambda=1-\frac1n$, and as in \Cref{thm:prob_nocap_nozero}, we let $\xirn$ and $\xiln$ be $\Geomnonneg{1-\lambda}$ random variables, independent of all SRWs and BRWs concerned.

\subsection{Green's functions}\label{sec:Greens_functions}

Let $g$ denote the Green's function of a simple random walk on $\Z^d$ (which is also the Green's function of a walk indexed by a critical Galton-Watson tree), and let $G$ denote the Green's function of the future of a branching random walk on $\Z^d$, i.e.\ let $G(x)=\sum_{k\ge0}\pr{\T_k=x}$. For two functions $f$ and $h$ on $\Z^d$ let $(f\star h)(z)$ denote the convolution $\sum_{x\in\Z^d}f(x)h(z-x)$. Then the following are known.

\begin{lemma}\label{lem:Greens_function}
For $d>2$ and any $\eps>0$ we have 
\begin{align}\label{eq:asymp_g}
g(z)\quad=\quad c_g\frac{1}{||z||^{d-2}}+o\left(\frac{1}{||z||^{d-\eps}}\right)\qquad\text{as }||z||\to\infty,
\end{align}
where $c_g$ is a constant depending on the dimension.

For $d>4$ and any $\eps>0$ we have
\begin{align}\label{eq:asymp_G}
G(z)\quad&=\quad\frac{\sigma^2}{2}(g\star g)(z) +\left(1-\frac{\sigma^2}{2}\right)g(z)-\frac{\sigma^2}{2}\E{g(X_1+z)}\\
\nonumber
\quad&=\quad\frac{\sigma^2}{2}(g\star g)(z)+O\left(g(z)\right)\quad=\quad c_G\frac{1}{||z||^{d-4}}+o\left(\frac{1}{||z||^{d-2-\eps}}\right)\quad\text{as }||z||\to\infty,
\end{align}
where $c_G$ is a constant depending on the dimension and $\sigma^2$.

For $d>6$ and any $\eps>0$ we have
\begin{align}\label{eq:asymp_G_star_g}
(G\star g)(z)\quad&=\quad\frac{\sigma^2}{2}(g\star g\star g)(z)+O((g\star g)(z))\quad\\
\nonumber
&=\quad c_{G\star g}\frac{1}{||z||^{d-6}}+o\left(\frac{1}{||z||^{d-4-\eps}}\right)&\qquad\text{as }||z||\to\infty,
\end{align}
where $c_{G\star g}$ is a constant depending on the dimension and $\sigma^2$.
\end{lemma}

See e.g.\ \cite[Theorem 1.5.4]{intersections_of_RWs} for~\eqref{eq:asymp_g}. Then~\eqref{eq:asymp_G} follows from a short calculation, using that $\E{d_i^+}=\frac{\sigma^2}{2}$ for $i\ge1$, $\E{d_0^+}=1$ and that $\GW(\mu)$ has the same Green's function as a SRW, and~\eqref{eq:asymp_G_star_g} follows from a short calculation using \eqref{eq:asymp_g} and \eqref{eq:asymp_G}.

For $\alpha\in(0,1)$ write $g_{\alpha}$ for the Green's function of a SRW ran for $\Geomnonneg{\alpha}$ steps, i.e.\ let $g_{\alpha}(x):=\sum_{k\ge0}(1-\alpha)^k\pr{X_k=x}$.

\subsection{Shift invariance of BRWs and the magic formula}\label{sec:shift_invariance_magic_formula}

A very important property of BRWs is the following. 

\begin{lemma} \label{lem:shift_invariance}
Let $\T=(\T,\Scal)$ be a BRW as above, let $(\til{\T}_{i})_{i\in\Z}=(\T_{i+1})_{i\in\Z}$ be obtained by re-rooting the tree $\T$ at $\T_1$, and let $\til{\Scal}_i=\Scal_{i+1}-\Scal_1$ be obtained by shifting the walk $\Scal$. Then we have $(\til{\T},\til{\Scal})\eqdist(\T,\Scal)$.
\end{lemma}

An analogous shift-invariance property for $\T(-\infty,0]$ was stated in \cite[Proposition 2]{range_of_tree-indexed_RW} and the exact statement of \Cref{lem:shift_invariance} appears in \cite[after (2.2)]{Cap_of_BRW}.

By using a last exit decomposition and the translation invariance of $\T$ and $\til{\T}$, we can obtain a `magic formula' similar to~\cite[Theorem 2.2]{RWs_random_sets}, featuring the probability in \Cref{thm:prob_nocap_nozero}.

Let
\begin{align*}
\A_n:=&\quad\left\{\til{\T}(-\infty,0)\cap\T[-\xiln,\xirn]=\emptyset\right\},\qquad
\B_n:=\quad\left\{0\not\in\T(0,\xirn]\right\},\\
\Ell_n:=&\quad\sum_{k=-\xiln}^{\xirn}\Ell^{+}\left(\T_k\right)\quad\text{where}\quad\Ell^{+}(y)=\sum_{j=0}^{\infty}\1{\til{\T}_j=y},\\
U_n:=&\quad\sum_{i=0}^{\infty}\sum_{j=-\xiln}^{\xirn}\til{d}^+_i g(\Xtil_i,\T_j)
-\sum_{i=0}^{\infty}\sum_{j=-\xiln}^{\xirn}\til{d}^+_i\1{\Xtil_i=\T_j}.
\end{align*}

\begin{lemma}\label{lem:magic_formula}
We have
\[\E{\1{\A_n}\1{\B_n}\Ell_n}\quad=\quad1.\]
\end{lemma}

The proof is analogous to the proof of~\cite[Theorem 2.2]{RWs_random_sets} and presented below for completeness.

\begin{proof}
Let $y_0,...,y_m\in\Z^d$ and let $\til{\T}$ be a branching random walk on $\Z^d$. Then by considering $\sup\{\ell\ge0:\:\til{\T}(-\ell)\in\{y_0,...,y_m\}\}$, using that a.s.\ $\til{\T}$ intersects $\{y_0,...,y_m\}$ only finitely many times, and using the shift invariance of $\til{\T}$, we get the following last exit decomposition formula.
\begin{align}
\nonumber
1\quad\ge&\quad\pr{\til{\T}(-\infty,0]\cap\{y_0,...,y_m\}\ne\emptyset}\\
\nonumber
=&\quad\sum_{\ell=0}^{\infty}\sum_{i=0}^{m}\1{y_i\not\in\{y_{i+1},..,y_m\}}\pr{\til{\T}(-\ell)=y_i,\:\til{\T}(-\infty,-\ell)\cap\{y_0,..,y_m\}=\emptyset}\\
\nonumber
=&\quad\sum_{\ell=0}^{\infty}\sum_{i=0}^{m}\1{y_i\not\in\{y_{i+1},..,y_m\}}\pr{\til{\T}(\ell)=-y_i,\:\til{\T}(-\infty,0)\cap\{y_0-y_i,..,y_m-y_i\}=\emptyset}\\
\label{eq:last_exit_decomp}
=&\quad\sum_{i=0}^{m}\1{y_i\not\in\{y_{i+1},..,y_m\}}\E{\#\left((-y_i)\text{ in }\til{\T}[0,\infty)\right)\1{\til{\T}_{-}\cap\{y_0-y_i,..,y_m-y_i\}=\emptyset}}.
\end{align}

By the definition of $\A_n$, $\B_n$ and $\Ell_n$ we have
\begin{align*}
&\E{\1{\A_n}\1{\B_n}\Ell_n}\quad\\
&\qquad=\quad\sum_{m=0}^{\infty} \sum_{\substack{x_0,...,x_m\\ \text{with }x_0=0}}\sum_{j=0}^{m}\pr{\xiln=m-j,\xirn=j,\T[-m+j,j]=(x_0-x_j,...,x_m-x_j)}\\
&\qquad\qquad\cdot\1{x_j\not\in\{x_{j+1},...,x_m\}}\E{\sum_{k=0}^{m}\#\left((x_k-x_j)\text{ in }\til{\T}[0,\infty)\right)\1{\til{\T}_{-}\cap\{x_0-x_j,..,x_m-x_j\}=\emptyset}}.
\end{align*}
Using the shift invariance of $\T$, and that $\left(\xiln\mid\xiln+\xirn=m\right)\sim\Unif{\{0,1,...,m\}}$, we get that this is
\begin{align*}
&=\quad\sum_{m=0}^{\infty}\sum_{\substack{x_0,...,x_m\\ \text{with }x_0=0}}\sum_{j=0}^{m}\frac{1}{m+1}\pr{\xiln+\xirn=m}\pr{\T[0,m]=(x_0,...,x_m)}\\
&\hspace{3cm}\cdot\1{x_j\not\in\{x_{j+1},...,x_m\}}\E{\sum_{k=0}^{m}\#\left((x_k-x_j)\text{ in }\til{\T}[0,\infty)\right)\1{\til{\T}_{-}\cap\{x_0-x_j,...,x_m-x_j\}=\emptyset}}.
\end{align*}
Then we use~\eqref{eq:last_exit_decomp} for $x_0-x_k,...,x_m-x_k$ in place of $y_0,...,y_m$ to get that this is
\begin{align*}
=\quad&\sum_{m=0}^{\infty}\sum_{\substack{x_0,...,x_m\\ \text{with }x_0=0}}\frac{1}{m+1}\pr{\xiln+\xirn=m}\pr{\T[0,m]=(x_0,...,x_m)}\\
&\hspace{3cm}\cdot\sum_{k=0}^{m}\pr{\til{\T}(-\infty,0]\cap\{x_0-x_k,...,x_m-x_k\}\ne\emptyset}.
\end{align*}
Note that $0\in\til{\T}(-\infty,0]\cap\{x_0-x_k,...,x_m-x_k\}$ for each $k$, hence\\ $\pr{\til{\T}(-\infty,0]\cap\{x_0-x_k,...,x_m-x_k\}\ne\emptyset}=1$. From this we get that the above sum is 1. This finishes the proof.
\end{proof}

We can simplify the formula in \Cref{lem:magic_formula} by removing some of the randomness of $\Ell_n$ as follows.

\begin{corollary}\label{cor:magic_formula}
We have
\begin{align}\label{eq:magic_formula}
\E{\1{\A_n}\1{\B_n}U_n}\quad=\quad1.
\end{align}
\end{corollary}

\begin{proof}
Use that $U_n=\econd{\Ell_n}{\T[-\xiln,\xirn],\Xtil,(\til{d}^+_i),\til{\T}(-\infty,0)}$, while $\A_n$ and $\B_n$ are measurable with respect to the $\sigma$-algebra in the conditioning.
\end{proof}

\subsection{Depth-first queue process}\label{sec:DFQP}

A useful tool for some of the later calculations is to encode trees via their depth-first queue process (DFQP). We start by recalling the definition and some important properties, and afterwards we present the notation we will use regarding this.

\subsubsection*{Definition and properties}

Given a rooted plane tree $t$ with $n$ edges, let $v_j$ ($j\in\{0,1,...,n\}$) be the $j$th new vertex encountered when going around the contour from the root anticlockwise. Let $\xi_j$ be the number of offsprings of $v_j$.

Then the depth-first queue process of $t$ is $W:\{0,1,...,n+1\}\to\Z$ such that $W_k=\sum_{j=0}^{k-1}(\xi_j-1)$. Note that in case of a $\GW(\mu)$ tree this is a random walk with iid $\eqdist(\mu-1)$ increments.

We say that $j\in[a,b]$ is a right minimum of $W$ on $[a,b]$ if $W_j=\min_{i\in[j,b]}W_i$, and we say that $j\in[a,b]$ is a record of $W$ on $[a,b]$ if $W_j=\max_{i\in[a,j]}W_i$. (If $[a,b]$ is not specified, we take $[a,b]=[0,j]$.)

Then the graph distance of $v_k$ from the root $v_0$ is the number of right minima of $W$ on the interval $[0,k]$, not counting the right minimum at $k$. The ancestors of $v_k$ are exactly the vertices corresponding to these right minima. (See \Cref{fig:DFQP_right_min} for an illustration.)

Note that we have $(W_j)_{j\in[0,k]}\eqdist(W_k-W_{k-j})_{j\in[0,k]}$, and $j$ being a right minimum on $[a,b]$ in the former process corresponds to $(k-j)$ being a record on $[k-b,k-a]$ in the latter one.

By flipping $W$ on $[0,k]$ we will mean considering the process $(W_k-W_{k-j})_{j\in[0,k]}$ instead of $(W_j)_{j\in[0,k]}$.

Given a walk $W$, the $j$th ladder width and ladder height are $\LW_j=R_j-R_{j-1}$ and $\LH_j=W_{R_j}-W_{R_{j-1}}$, respectively, where $R_j$ denotes the $j$th record. (See \Cref{fig:flipped_walk_LW_LH} for an illustration.) Note that for a walk $W$ with iid increments, the parts $(W_i)_{i\in[R_{j-1},R_j]}$ are iid, and in particular the pairs $(\LW_j,\LH_j)$ are iid.

\begin{figure}%[h]
\centering
\begin{subfigure}{0.6\linewidth}
\includegraphics[width=0.9\linewidth]{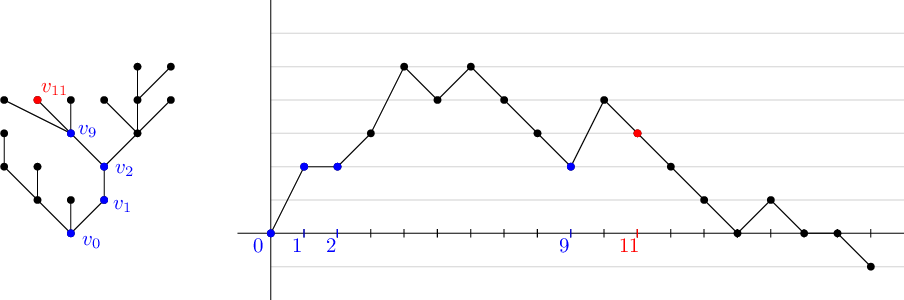}
\caption{A tree and its DFQP. The ancestors of vertex $v_{11}$ correspond to the right minima on [0,11]. \\}
\label{fig:DFQP_right_min}
\end{subfigure}
\begin{subfigure}{0.01\linewidth}
\end{subfigure}
\begin{subfigure}{0.38\linewidth}
\includegraphics[width=0.8\linewidth]{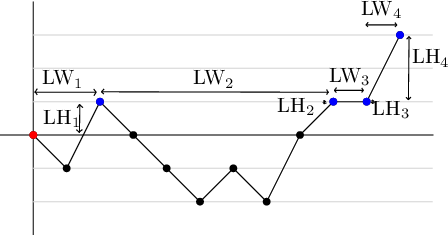}
\caption{The previous walk flipped on [0,11]. The right minima on [0,11] become records. The ladder widths and ladder heights are also marked.}
\label{fig:flipped_walk_LW_LH}
\end{subfigure}
\caption{}
\end{figure}

Also note that if $W$ is a random walk with iid $\eqdist(\mu-1)$ increments, then the parts $(W_{i}+k)_{i\in[\sigma_{-k},\sigma_{-k-1})}$ of the walk (for $k\ge0$, where $\sigma_{-k}=\min\{n:\:W_n=-k\}$) define independent $\GW(\mu)$ trees.

\subsubsection*{Notation}

In what follows we let $W$ denote a random walk with iid $\sim(\mu-1)$ increments, and we denote its ladder widths and ladder heights as $(\LW_i)$ and $(\LH_i)$ respectively. We sometimes write $(\LW,\LH)$ for a pair of random variables distributed as $(\LW_1,\LH_1)$. We also sometimes write $\GW$ to denote a $\GW(\mu)$ tree or $(\GW_i)$ to denote a sequence of independent $\GW(\mu)$ trees. We write $|\GW|$ to mean the number of vertices in $\GW$.

The later calculations will often feature quantities like $\E{\lambda^{\LW}}$ or $\E{\lambda^{|\GW|}}$, we will introduce notation for these the first time they are used.

\subsection{Other notation}\label{sec:other_notation}

For functions $a,b:\Znonneg\to[0,\infty)$ we say that $a(n)\lesssim b(n)$ as $n\to\infty$ if there exists constant $C>0$ such that for all sufficiently large $n$ we have $a(n)\le Cb(n)$. We say that $a(n)\ll b(n)$ as $n\to\infty$ if for any constant $c>0$, for all sufficiently large $n$ (in terms of $c$) we have $a(n)\le cb(n)$. We define $\gtrsim$ and $\gg$ analogously. We say that $a(n)\asymp b(n)$ as $n\to\infty$ if we have $a(n)\lesssim b(n)\lesssim a(n)$ as $n\to\infty$. We say that $a(n)\sim b(n)$ as $n\to\infty$ if we have $|a(n)-b(n)|\ll a(n)$. 
Unless specified otherwise, these relations will be as $n\to\infty$.

We will sometimes write things like $\lesssim_{x,y}$ to emphasise that the implicit constant in $\lesssim$ is a function of parameters $x$ and $y$. Almost all constants will depend on $\sigma^2$, so we will drop this from the notation.

\section{Proof of the concentration of $U_n$}\label{sec:E_and_Var_of_U}

Recall the definition of $U_n$ from \Cref{sec:shift_invariance_magic_formula}. The goal of this section is to establish the following results.

\begin{proposition}\label{pro:E_Un}In $d=8$ dimensions we have
\[\E{U_n}\quad\sim\quad\frac{\log n}{c_8},\]
where $c_8$ is a positive constant depending only on $\sigma^2$.
\end{proposition}

\begin{proposition}\label{pro:Var_Un}In $d=8$ dimensions we have
\[\Var{U_n}\quad\lesssim\quad\log n,\]
where the implicit constants in $\lesssim$ only depend on $\sigma^2$.
\end{proposition}

Instead of computing the expectation and variance of $U_n$ directly, it is helpful to first consider the following quantities.
\begin{align*}
G_n:=\quad\sum_{j=-\infty}^{\infty}\1{j\in[-\xiln,\xirn]}G\left(\T_j\right),\qquad G_n^{+}:=\quad\sum_{j=1}^{\infty}\1{j\in[1,\xirn]}G\left(\T_j\right),\qquad G_n^{-}:=G_n-G_n^{+}.
\end{align*}

Note that $G_n=\econd{U_n}{\T[-\xiln,\xirn]}$, hence we have
\begin{align*}
\E{U_n}\quad=&\quad\E{G_n},\\
\Var{U_n}\quad=&\quad\Var{G_n}+\E{\varcond{U_n}{\T}},
\end{align*}
where we dropped $\xiln$ and $\xirn$ from the notation, as we often will in what follows.

For proving Propositions~\ref{pro:E_Un} and~\ref{pro:Var_Un}, it will be sufficient to show the following.
\begin{lemma}
In $d=8$ dimensions we have
\begin{align}
\label{eq:E_Gn}
\E{G_n}\quad\sim\quad\frac{\log n}{c_8},\\
\label{eq:Var_Gn}
\Var{G_n}\quad\lesssim\quad\log n,\\
\label{eq:E_Varcond_Un}
\E{\varcond{U_n}{\T}}\quad\lesssim\quad\log n,
\end{align}
where $c_8$ is a positive constant depending only on $\sigma^2$.
\end{lemma}

\subsection{Estimates for one $\GW$-tree}\label{sec:GW_sums}

First we compute some expectations regarding only one $\GW$-tree. We will make use of these in expressing the expectation and variance of $G_n^+$ and $G_n^-$.

We let $\lambda=1-\frac1n$ as in \Cref{sec:further_notation_BRW}, define $\GW$, $\LW$ and $\LH$ as in \Cref{sec:DFQP}, and let
\begin{align}\label{eq:theta_def}
\theta:=\E{\lambda^{\LW}},\quad\til{\theta}:=\E{\lambda^{|\GW|}},\quad\what{\theta}:=\E{\lambda^{\LW}\til{\theta}^{\LH}},\quad\dot{\theta}:=\E{\til{\theta}^{\LH}}.
\end{align}

\begin{lemma}\label{lem:GW_sums}
Let $T=\GW$ be a $\GW(\mu)$ tree with its vertices indexed by $0$, $1$, ... $|\GW|-1$ in the depth-first order, and let $S$ be a random walk on $\Z^d$ indexed by $T$. Let $\lambda=1-\frac1n$, and let $h$ and $h_2$ be some functions $\Z^d\to[0,\infty]$ and $\Z^d\times\Z^d\to[0,\infty]$ respectively. Then we have
\begin{align}
\label{eq:GW_one_sum_fwd}
\E{\sum_{i=0}^{|\GW|-1}\lambda^{i}h(S_i)}\quad=&\quad \sum_{j\ge0}\E{h(X_j)}\theta^j,\\
\label{eq:GW_one_sum_bwd}
\E{\sum_{i=0}^{|\GW|-1}\lambda^{|\GW|-i}h(S_i)}\quad=&\quad \sum_{j\ge0}\E{h(X_j)}\dot{\theta}^j\til{\theta},\\
\label{eq:GW_one_sum_all}
\E{\lambda^{|\GW|}\sum_{i=0}^{|\GW|-1}h(S_i)}\quad=&\quad \sum_{j\ge0}\E{h(X_j)}\what{\theta}^j\til{\theta},
\end{align}
\begin{align}
\label{eq:GW_two_sums_fwd}
&\E{\sum_{0\le i<j<|\GW|} \lambda^{j}h_2(S_i,S_j)}\\
\nonumber
&\qquad\quad=\quad \sum_{k\ge1}\sum_{r\ge1}\sum_{\ell\ge0}\E{h_2(X_{\ell}+\til{X}_{r},X_{\ell}+\what{X}_{k})}{\theta}^{k}{\theta}^{\ell}\what{\theta}^{r}\frac{\til{\theta}(\theta-\what{\theta})}{\theta\what{\theta}(1-\til{\theta})}\\
\nonumber
&\qquad\quad\quad+\quad\sum_{k\ge1}\sum_{\ell\ge0}\E{h_2(X_{\ell},X_{\ell}+\what{X}_{k})}{\theta}^{k}{\theta}^{\ell},\\
\label{eq:GW_two_sums_bwd}
&\E{\sum_{0\le i<j<|\GW|} \lambda^{|\GW|-i}h_2(S_i,S_j)}\\
\nonumber
&\qquad\quad=\quad
\sum_{k\ge1}\sum_{r\ge1}\sum_{\ell\ge0}\E{h_2(X_{\ell}+\til{X}_{r},X_{\ell}+\what{X}_{k})}\what{\theta}^{k-1}\dot{\theta}^{\ell+r-1}\til{\theta}\E{\til{\theta}^{\LH}\LH}\\
\nonumber
&\qquad\qquad+\quad\sum_{k\ge1}\sum_{\ell\ge0}\E{h_2(X_{\ell},X_{\ell}+\what{X}_{k})}\what{\theta}^{k}\dot{\theta}^{\ell}\til{\theta}.
\end{align}
where $X$, $\til{X}$ and $\what{X}$ are independent SRWs on $\Z^d$, and $\theta$, $\til{\theta}$, $\what{\theta}$ and $\dot{\theta}$ are defined as in~\eqref{eq:theta_def}.
\end{lemma}

\begin{proof}[Proof of~\eqref{eq:GW_one_sum_fwd}, \eqref{eq:GW_one_sum_bwd} and \eqref{eq:GW_one_sum_all}] First we present the proof of \eqref{eq:GW_one_sum_fwd}.

Note that the value of $S_i$ is distributed as $X_{d(T_i,T_0)}$, where $X$ is a SRW independently of $T$, and $d(T_i,T_0)$ is the graph distance between the $i$th vertex and the root of $T$. (If $i\ge|\GW|$ then both sides are undefined.)

Let $W$ be a random walk with iid $\eqdist(\mu-1)$ increments and let $X$ be a SRW on $\Z^d$, independent of it.
Then by interpreting $W$ as the DFQP of $T$, expressing $d(T_i,T_0)$ in terms of $W$ and using the above observation we get that
\begin{align*}
&\E{\sum_{i=0}^{|\GW|-1}\lambda^{i}h(S_i)}\quad=\quad\sum_{i\ge0}\lambda^i\E{\1{i<|\GW|}h(S_i)}\\
&=\quad\sum_{i\ge0}\lambda^i\E{\1{\:0\text{ is a right min of }W\text{ on }[0,i]}\:h(X_{\#\left(\text{right min of }W\text{ on }[0,i]\text{, excl. }i\right)})}.
\end{align*}

Recall the definition of flipping and the correspondence between right minima and records from \Cref{sec:DFQP}.
Flipping $W$ on $[0,i]$, and letting $i=\sum_{s=1}^{j}\LW_s$ be the $j$th record we get that the above sum is
\begin{align*}
&=\quad\sum_{i\ge0}\lambda^i\E{\1{\:i\text{ is a record of }W}\: h(X_{\#\left(\text{records of }W\text{ on }[0,i]\text{, excl. }0\right)})}\\
&=\quad\sum_{j\ge0}\E{h(X_{j})}\E{\lambda^{\sum_{s=1}^{j}\LW_s}}\quad=\quad\sum_{j\ge0}\E{h(X_{j})}\E{\lambda^{\LW}}^j.
\end{align*}

The proofs of \eqref{eq:GW_one_sum_bwd} and \eqref{eq:GW_one_sum_all} are analogous to the proof of \eqref{eq:GW_one_sum_fwd}, the only new observation needed is the following. Given $(W_s)_{s\in[0,i]}$ such that $0$ is a right minimum on $[0,i]$, the number of remaining vertices in the tree after the $i$th one is distributed as $\left(|\GW|-i\mid(W_s)_{s\in[0,i]}\right)\eqdist\left(\sigma_{-1}-i\mid(W_s)_{s\in[0,i]}\right)\eqdist\sum_{j=1}^{W_i+1}|\GW_j|$.

Using this, we can write
\begin{align*}
&\E{\sum_{i=0}^{|\GW|-1}\lambda^{|\GW|-i}h(S_i)}\quad=\quad\sum_{i\ge0}\E{\1{i<|\GW|}\lambda^{|\GW|-i}h(S_i)}\\
&=\quad\sum_{i\ge0}\E{\1{\:0\text{ is a right min of }W\text{ on }[0,i]}\:h(X_{\#\left(\text{right min of }W\text{ on }[0,i]\text{, excl. }i\right)})\:\til{\theta}^{W_i+1}}\\
&=\quad\sum_{i\ge0}\E{\1{\:i\text{ is a record of }W}\:h(X_{\#\left(\text{records of }W\text{ on }[0,i]\text{, excl. }0\right)})\:\til{\theta}^{W_i+1}}\\
&=\quad\sum_{j\ge0}\:\E{h(X_j)}\E{\til{\theta}^{\LH}}^j\til{\theta}
\end{align*}
which concludes the proof of \eqref{eq:GW_one_sum_bwd}.
The proof of \eqref{eq:GW_one_sum_all} is analogous, but the exponential term is $\lambda^{|\GW|}$ in the first line, $\lambda^i\:\til{\theta}^{W_i+1}$ in the second and third line, and $\E{\lambda^{\LW}\til{\theta}^{\LH}}\til{\theta}$ in the last line.
\end{proof}

\begin{proof}[Proof of~\eqref{eq:GW_two_sums_fwd} and~\eqref{eq:GW_two_sums_bwd}] We start by presenting the proof of \eqref{eq:GW_two_sums_fwd}.
As above, let $W$ be a random walk with iid $\eqdist(\mu-1)$ increments. The records, etc below will all refer to $W$. Let $X$, $\til{X}$ and $\what{X}$ be SRWs on $\Z^d$, independent of each other and $W$.

Note that we have $(S_i,S_j)\eqdist(X_{d(T_i\wedge T_j,T_0)}+\til{X}_{d(T_i\wedge T_j,T_i)},X_{d(T_i\wedge T_j,T_0)}+\what{X}_{d(T_i\wedge T_j,T_j)})$, where $X$, $\til{X}$ and $\what{X}$ are independent of $T$, and $T_i\wedge T_j$ denotes the most recent common ancestor of $T_i$ and $T_j$.

Then by interpreting $W$ as the DFQP of $T$, using the above observation and expressing the graph distances in terms of $W$, we get that
\begin{align*}
&\E{\sum_{0\le i<j<|\GW|} \lambda^{j}h_2(S_i,\:S_j)}\quad=\quad\sum_{i\ge0}\sum_{k\ge1}\lambda^{i+k}\E{\1{i+k<|\GW|}h_2(S_i,\:S_{i+k})}\\
&=\quad\sum_{i\ge0}\sum_{k\ge1}\lambda^{i+k}\mathbb{E}\Bigg[\1{\:0\text{ right min. on }[0,i+k]}\cdot
h_2\Bigg(X_{\substack{\#(\text{right min on both }[0,i]\\ \text{ and }[0,i+k]\text{, excl. }0)}}
+\til{X}_{\substack{\#(\text{right min on }[0,i]\\ \text{ but not }[0,i+k])}}\:,\\
&\hspace{7cm} X_{\substack{\#(\text{right min on both }[0,i]\\ \text{ and }[0,i+k]\text{, excl. }0)}}
+\what{X}_{\substack{\#(\text{right min on }[0,i+k]\\ \text{ but not }[0,i])}}\Bigg)\Bigg].
\end{align*}
Flipping $W$ on $[0,i+k]$ we get that this is
\begin{align*}
&=\quad\sum_{i\ge0}\sum_{k\ge1}\lambda^{i+k}\mathbb{E}\Bigg[\1{\:i+k\text{ is a record}}\cdot
h_2\Bigg(X_{\substack{\#(\text{records on both }[k,i+k]\\ \text{ and }[0,i+k],\text{ excl. }i+k)}}
+\til{X}_{\substack{\#(\text{records on }[k,i+k]\\ \text{ but not }[0,i+k])}}\:,\\
&\hspace{7cm}X_{\substack{\#(\text{records on both }[k,i+k]\\ \text{ and }[0,i+k],\text{ excl. }i+k)}}
+\what{X}_{\substack{\#(\text{records on }[0,i+k]\\ \text{ but not }[k,i+k])}}\Bigg)\Bigg].
\end{align*}
Considering $j$ and $\ell$ so that $k\in(R_j,R_{j+1}]$, $i+k=R_{j+\ell}$ (recall that $R_m=\sum_{s=1}^{m}\LW_s$ denotes the $m$th record) we get that the above sum is
\begin{align} \label{eq:GW_two_fwd_sums_rephrased}
&=\:\sum_{j\ge0}\sum_{\ell\ge1}\E{\lambda^{R_{j+\ell}}\sum_{k=R_j+1}^{R_{j+1}}h_2\left(X_{\ell-1}+\til{X}_{\#(\text{records in }[k,R_{j+1}-1])},\:X_{\ell-1}+\what{X}_{j+1}\right)}\\
\nonumber
&=\:\sum_{j\ge0}\sum_{\ell\ge1}\E{\lambda^{\LW}}^{j+\ell-1} \E{\lambda^{\LW_1}\sum_{k=1}^{\LW_1}h_2\left(X_{\ell-1}+\til{X}_{\#(\text{records in }[k,\LW_1-1])},\:X_{\ell-1}+\what{X}_{j+1}\right)}.
\end{align}
(See \Cref{fig:GW_sum_proof} for an illustration of the above steps.)

\begin{figure}
\centering
\begin{subfigure}{0.18\linewidth}
\includegraphics[width=\linewidth]{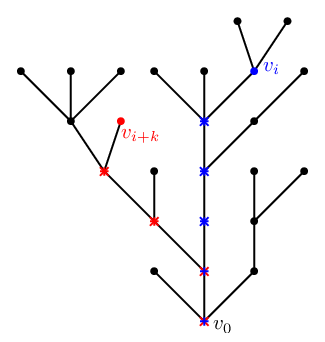}
\label{fig:GW_sum_proof_GW}
\caption{Example of a GW tree with $v_i$ and $v_{i+k}$ and their ancestors marked. \\ \\}
\end{subfigure}
\begin{subfigure}{0.01\linewidth}
\end{subfigure}
\begin{subfigure}{0.01\linewidth}
\end{subfigure}
\begin{subfigure}{0.43\linewidth}
\includegraphics[width=\linewidth]{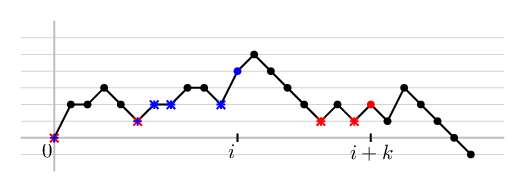}
\label{fig:GW_sum_proof_DFQP}
\caption{The DFQP corresponding to this tree. The ancestors of $v_i$ and $v_{i+k}$ correspond to the right minima on $[0,i]$ and $[0,i+k]$ respectively. \\ \\ \\}
\end{subfigure}
\begin{subfigure}{0.01\linewidth}
\end{subfigure}
\begin{subfigure}{0.01\linewidth}
\end{subfigure}
\begin{subfigure}{0.34\linewidth}
\includegraphics[width=\linewidth]{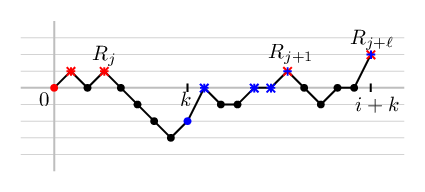}
\label{fig:GW_sum_proof_flipped_walk}
\caption{The DFQP flipped on $[0,i+k]$. The ancestors of $v_{i+k}$ and $v_i$ correspond to the records on $[0,i+k]$ and $[k,i+k]$ respectively. We consider $j$ and $\ell$ so that $i+k$ is the $(j+\ell)$th record, and $k$ falls between the $j$th and $(j+1)$th record.}
\end{subfigure}
\caption{}
\label{fig:GW_sum_proof}
\end{figure}

To deal with the second expectation, we note the following. For $r\ge1$ we have
\begin{align}
\nonumber
&\E{\lambda^{\LW_1}\#\left(k\in[1,\LW_1]:\text{exactly }r\text{ records in }[k,\LW_1-1]\right)}\\
\nonumber
&\quad=\quad\sum_{\substack{s,u,(w_0,...,w_{s+u}):\\ \LW_1(w)=s+u \\ \text{exactly }r\text{ records in }[s,s+u-1]}}
\nonumber
\lambda^{s+u}\cdot\pr{W\text{ starts as }(w_0,...,w_{s+u})}\\
\nonumber
&\quad=\quad\sum_{\substack{L_1,...,L_r,\\ H_1,...,H_r}} \pr{\LW_i=L_i,\LH_i=H_i\text{ for }i=1,...,r}\lambda^{\sum_{i=1}^{r}L_i}\\
\label{eq:count_k_with_r_records}
&\hspace{3cm}\cdot\sum_{\substack{s,(w_0,...,w_{s}):\\ w_1,...,w_s<0 \\ w_s\in\left[-\sum_{i=1}^{r}H_i,-\sum_{i=1}^{r-1}H_i-1\right]}}
\lambda^{s}\cdot\pr{W\text{ starts as }(w_0,...,w_{s})}.
\end{align}

By flipping $w$ on $[0,s]$, we can rephrase the second sum in~\eqref{eq:count_k_with_r_records} as
\begin{align*}
\sum_{\substack{s,(w_0,...,w_{s}):\\ s\text{ a strict min on }[0,s] \\ w_s\in\left[-\sum_{i=1}^{r}H_i,-\sum_{i=1}^{r-1}H_i-1\right]}}
\lambda^{s}\cdot\pr{W\text{ starts as }(w_0,...,w_{s})}.
\end{align*}
Then using that $W$ hits each negative integer as a strict minimum exactly once, and it takes independent $\eqdist|\GW|$ number of steps between them, we get that this is
\begin{align*}
=\quad\sum_{m=\sum_{i=1}^{r-1}H_i+1}^{\sum_{i=1}^{r}H_i}\E{\lambda^{|\GW|}}^m\quad=\quad\til{\theta}^{H_1+...+H_{r-1}+1}\cdot\frac{1-\til{\theta}^{H_r}}{1-\til{\theta}}.
\end{align*}

(See \Cref{fig:GW_sum_proof_2} for an illustration of the above steps.)

\begin{figure}
\centering
\begin{subfigure}{0.57\linewidth}
\includegraphics[width=\linewidth]{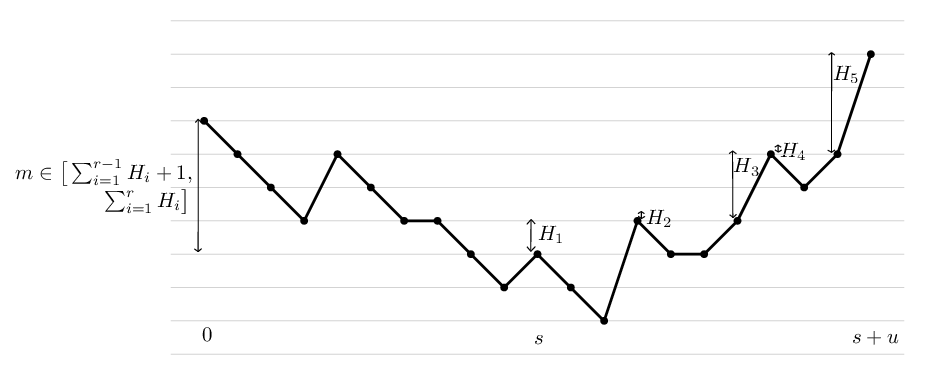}
\caption{An illustration of $(w_0,...,w_{s+u})$ with $\LW_1=s+u$ and $(w_s,...,w_{s+u})$ having exactly $r=5$ records. If these $r$ records are of heights $H_1$, ..., $H_r$ then $(w_0,...,w_s)$ should end at a height $-m\in\left[-\sum_{i=1}^{r-1}H_i+1,-\sum_{i=1}^{r}H_i\right]$, and it should have 0 as a strict maximum.}
\label{fig:GW_sum_proof_k_in_valley}
\end{subfigure}
\begin{subfigure}{0.1\linewidth}
\end{subfigure}
\begin{subfigure}{0.1\linewidth}
\end{subfigure}
\begin{subfigure}{0.4\linewidth}
\includegraphics[width=\linewidth]{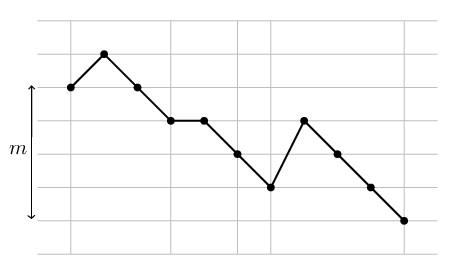}
\caption{Flipping $w$ on $[0,s]$ we get a walk that ends at $-m$ and has $-m$ as a strict minimum. This walk decomposes into $m$ DFQPs.\\ }
\label{fig:GW_sum_proof_descending_walk}
\end{subfigure}
\caption{}
\label{fig:GW_sum_proof_2}
\end{figure}

This shows that~\eqref{eq:count_k_with_r_records} is
\begin{align*}
=\quad\frac{\til{\theta}}{1-\til{\theta}}\:\E{\lambda^{\LW}\til{\theta}^{\LH}}^{r-1}\left(\E{\lambda^{\LW}}-\E{\lambda^{\LW}\til{\theta}^{\LH}}\right).
\end{align*}

This way we obtain the coefficient of $\E{h_2\left(X_{\ell-1}+\til{X}_r,\:X_{\ell-1}+\what{X}_{j+1}\right)}$ in~\eqref{eq:GW_two_fwd_sums_rephrased} for $r\ge1$.
We find the coefficient for $r=0$ by noting that for any realisation of $(W_s)_{s\le\LW_1}$, the only $k\in[1,\LW_1]$ such that there are no records in $[k,\LW_1-1]$ is $k=\LW_1$.

This finishes the proof of \eqref{eq:GW_two_sums_fwd}. The proof of \eqref{eq:GW_two_sums_bwd} is similar, we present the main steps below.

Firstly, we express the sum in terms of $W$ and we use that given $(W_s)_{s\le i+k}$ such that 0 is a right minimum on $[0,i+k]$, the number of vertices of $\GW$ after the $(i+k)$th one is $\eqdist\sum_{s=1}^{W_{i+k}+1}|\GW_s|$. We see that
\begin{align*}
&\E{\sum_{0\le i<j<|\GW|} \lambda^{|\GW|-i}h_2(S_i,S_j)}\\
&=\quad\sum_{i\ge0}\sum_{k\ge1}\mathbb{E}\Bigg[\1{\:0\text{ right min. on }[0,i+k]}\cdot\lambda^k\cdot\E{\lambda^{|\GW|}}^{W_{i+k}+1}\cdot\\
&h_2\Bigg(X_{\substack{\#(\text{right min on both }[0,i]\\ \text{ and }[0,i+k]\text{, excl. }0)}}
+\til{X}_{\substack{\#(\text{right min on }[0,i]\\ \text{ but not }[0,i+k])}},X_{\substack{\#(\text{right min on both }[0,i]\\ \text{ and }[0,i+k]\text{, excl. }0)}}
+\what{X}_{\substack{\#(\text{right min on }[0,i+k]\\ \text{ but not }[0,i])}}\Bigg)\Bigg].
\end{align*}
Flipping $W$ on $[0,i+k]$ we get that this is
\begin{align*}
&=\quad\til{\theta}\:\sum_{i\ge0}\sum_{k\ge1}\mathbb{E}\Bigg[\1{\:i+k\text{ is a record}}\cdot\lambda^{k}\til{\theta}^{W_{i+k}}\cdot
h_2\Bigg(X_{\substack{\#(\text{records on both }[k,i+k]\\ \text{ and }[0,i+k],\text{ excl. }i+k)}}
+\til{X}_{\substack{\#(\text{records on }[k,i+k]\\ \text{ but not }[0,i+k])}}\:,\\
&\hspace{8.2cm}X_{\substack{\#(\text{records on both }[k,i+k]\\ \text{ and }[0,i+k],\text{ excl. }i+k)}}
+\what{X}_{\substack{\#(\text{records on }[0,i+k]\\ \text{ but not }[k,i+k])}}\Bigg)\Bigg].
\end{align*}
As in the proof of \eqref{eq:GW_two_sums_fwd} we consider $j$ and $\ell$ such that $k\in(R_j,R_{j+1}]$, $i+k=R_{j+\ell}$ to get that the above sum is
\begin{align}\label{eq:GW_two_bwd_sums_rephrased}
&=\:\til{\theta}\:\sum_{j\ge0}\sum_{\ell\ge1}\E{\til{\theta}^{W_{R_{j+\ell}}}\sum_{k=R_j+1}^{R_{j+1}}\lambda^kh_2\left(X_{\ell-1}+\til{X}_{\#(\text{records in }[k,R_{j+1}-1])},\:X_{\ell-1}+\what{X}_{j+1}\right)}\\
\nonumber
&=\:\til{\theta}\:\sum_{j\ge0}\sum_{\ell\ge1}\E{\til{\theta}^{\LH}\lambda^{\LW}}^{j}\E{\til{\theta}^{\LH}}^{\ell-1}\\
\nonumber
&\hspace{3cm}\cdot\E{\til{\theta}^{\LH_1}\sum_{k=1}^{\LW_1}\lambda^kh_2\left(X_{\ell-1}+\til{X}_{\#(\text{records in }[k,\LW_1-1])},\:X_{\ell-1}+\what{X}_{j+1}\right)}.
\end{align}
Analogously to before, for each $r\ge0$ we compute the contribution to the last expectation from terms with exactly $r$ records in $[k,\LW_1-1]$. For $r\ge1$ we use a representation analogous to \eqref{eq:count_k_with_r_records}. We write
\begin{align}
\nonumber
&\E{\til{\theta}^{\LH_1}\sum_{k=1}^{\LW_1}\lambda^k\1{\#\left(\text{records in }[k,\LW_1-1]\right)=r}}\\
\nonumber
&\quad=\quad\sum_{\substack{s,u,(w_0,...,w_{s+u}):\\ \LW_1(w)=s+u \\ \text{exactly }r\text{ records in }[s,s+u-1]}}
\nonumber
\lambda^{s}\til{\theta}^{w_{s+u}}\cdot\pr{W\text{ starts as }(w_0,...,w_{s+u})}\\
\nonumber
&\quad=\quad\sum_{H_1,...,H_r} \pr{\LH_i=H_i\text{ for }i=1,...,r}\\
\label{eq:terms_with_r_records}
&\hspace{2cm}\cdot\sum_{\substack{s,(w_0,...,w_{s}):\\ w_1,...,w_s<0 \\ w_s\in\left[-\sum_{i=1}^{r}H_i,-\sum_{i=1}^{r-1}H_i-1\right]}}
\lambda^{s}\til{\theta}^{w_s+\sum_{i=1}^{r}H_i}\cdot\pr{W\text{ starts as }(w_0,...,w_{s})}.
\end{align}
By flipping $w$ on $[0,s]$ and using that $W$ hits each nonnegative integer exactly once as a strict minimum, we get that the last sum equals
\begin{align*}
&\sum_{\substack{s,(w_0,...,w_{s}):\\ s\text{ a strict min on }[0,s] \\ w_s\in\left[-\sum_{i=1}^{r}H_i,-\sum_{i=1}^{r-1}H_i-1\right]}}
\lambda^{s}\til{\theta}^{w_s+\sum_{i=1}^{r}H_i}\cdot\pr{W\text{ starts as }(w_0,...,w_{s})}\\
&=\quad\sum_{m=\sum_{i=1}^{r-1}H_i+1}^{\sum_{i=1}^{r}H_i} \E{\lambda^{|\GW|}}^m\til{\theta}^{-m}\til{\theta}^{\sum_{i=1}^{r}H_i}\quad=\quad H_r\:\til{\theta}^{\sum_{i=1}^{r}H_i},
\end{align*}
hence \eqref{eq:terms_with_r_records} is equal to $\E{\til{\theta}^{\LH}}^{r-1}\E{\til{\theta}^{\LH}\LH}$. Plugging this back to \eqref{eq:GW_two_bwd_sums_rephrased} we find the coefficient of $\E{h_2\left(X_{\ell-1}+\til{X}_{r},\:X_{\ell-1}+\what{X}_{j+1}\right)}$ for $r\ge1$. We can find the coefficient for $r=0$ by noting that in this case the only contribution to \eqref{eq:terms_with_r_records} is from $k=\LW_1$. This finishes the proof of \eqref{eq:GW_two_sums_bwd}.
\end{proof}

\subsection{Expectation of $G_n^+$ and $G_n^-$}\label{sec:E_Gn}

In this section we prove~\eqref{eq:E_Gn}.

Let us partition the vertices of $\T(0,\infty)$ as follows. Let $\T^+_{k,\ell}$ be the tree consisting of the $\ell$th offspring on the future side from the $k$th spine vertex, and its descendants. (If $d^+_k<\ell$ then we consider this tree to be empty. Note that given $(d^+_k)$, the trees that are non-empty are independent $\GW$ trees.) Let
\[
Y^{+}_n(k,\ell):=\quad\sum_{j=1}^{\infty}\1{j\le\xirn}\1{j\in\T^+_{k,\ell}}G(\T_j)
\]
be the part of the sum in $G^+_n$ corresponding to the vertices in $\T^+_{k,\ell}$. Then we can write
\[
G^+_n\quad=\quad\sum_{k=0}^{\infty}\sum_{\ell=1}^{\infty}Y^{+}_n(k,\ell).
\]

Using the independence of the offspring numbers and the memoryless property of $\xirn$, we get that
\begin{align*}
\E{Y^{+}_n(k,\ell)}\:=\:\pr{d^+_k\ge\ell}\pr{\xirn\ge\sum_{i=1}^{\substack{d^+_0+...+d^+_{k-1}\\\qquad+\ell-1}}|\GW_i|+1}\E{\sum_{i=0}^{|\GW|-1}\lambda^iG(X_{k+1}+T_i)}
\end{align*}
where $T$ is a random walk indexed by $\GW$, and $X$ is an independent SRW. By~\eqref{eq:GW_one_sum_fwd} we know that this is
\[
=\quad
\begin{cases}
\pr{d^+_1\ge\ell}\theta_0\theta_1^{k-1}\til{\theta}^{\ell-1}\lambda\sum\limits_{j\ge0}\E{G(X_{j+k+1})}\theta^j&\text{ if }k\ge1\\
\pr{d^+_0\ge\ell}\til{\theta}^{\ell-1}\lambda\sum\limits_{j\ge0}\E{G(X_{j+1})}\theta^j&\text{ if }k=0,
\end{cases}
\]
where $\til{\theta}$ is as in \Cref{sec:GW_sums}, and
\begin{align}\label{eq:theta0_def}
\theta_0:=\E{\til{\theta}^{d_0^{+}}},\quad\theta_1:=\E{\til{\theta}^{d_1^{+}}}.
\end{align}

Summing this over all $k$ and $\ell$, and using that $\sum_{\ell\ge1}\pr{d^{+}_k\ge\ell}\til{\theta}^{\ell-1}=\E{1+\til{\theta}+...+\til{\theta}^{d^{+}_k-1}}=\frac{1-\E{\til{\theta}^{d^{+}_k}}}{1-\til{\theta}}$, we get that
\begin{align}
\label{eq:Gplus_decomposition}
\E{G^+_n}\quad=\quad\sum_{j\ge0}\lambda\theta^j\E{G(X_{j+1})}\frac{1-\theta_0}{1-\til{\theta}}\quad+\quad\sum_{j\ge0}\sum_{k\ge1}\lambda\theta_1^{k-1}\E{G(X_{j+k+1})}\theta_0\frac{1-\theta_1}{1-\til{\theta}}.
\end{align}

Using \Cref{lem:theta0_theta_1}, the first term in~\eqref{eq:Gplus_decomposition} is $\lesssim\sum_{j\ge0}\E{G(X_j)}=(G\star g)(0)\asymp1$.

Using \Cref{cor:theta_thetatil_thetadot} and the notation $g_{\alpha}$ from the end of \Cref{sec:Greens_functions}, the second term in~\eqref{eq:Gplus_decomposition} is
\begin{align}\label{eq:E_Gplus_sum}
\sim\quad\frac{1-\theta_1}{1-\til{\theta}}\sum_{x}\sum_{y}g_{1-\theta_1}(x)g(y)G(x+y)\quad=\quad\frac{1-\theta_1}{1-\til{\theta}}\sum_{x}g_{1-\theta_1}(x)(g\star G)(x).
\end{align}

From \Cref{lem:theta0_theta_1} we know that $1-\theta_1\asymp\frac{1}{\sqrt{n}}$, and from \Cref{lem:theta_fraction} we know that\\ $\frac{1-\theta_1}{1-\til{\theta}}\sim\frac{\sigma^2}{2}$. Writing $\alpha=1-\theta_1$, we can upper bound the sum in~\eqref{eq:E_Gplus_sum} as
\begin{align*}
\sum_{x}g_{\alpha}(x)(g\star G)(x)\:\le \sum_{x:\:||x||\le\frac{1}{\sqrt{\alpha}}\log\left(\frac{1}{\sqrt{\alpha}}\right)}g(x)(g\star G)(x)+ \sum_{x:\:||x||>\frac{1}{\sqrt{\alpha}}\log\left(\frac{1}{\sqrt{\alpha}}\right)}g(x)(g\star G)(x).
\end{align*}

Here
\begin{align*}
&\sum_{x:\:||x||\le\frac{1}{\sqrt{\alpha}}\log\left(\frac{1}{\sqrt{\alpha}}\right)}g(x)(g\star G)(x)\quad\sim\sum_{x:\:||x||\le\frac{1}{\sqrt{\alpha}}\log\left(\frac{1}{\sqrt{\alpha}}\right)}\frac{c_gc_{G\star g}}{||x||^8}\quad\\
&\qquad\sim\quad c_{g}c_{G\star g}\cdot\frac{2\pi^{\frac{d}{2}}}{\Gamma\left(\frac{d}{2}\right)}\cdot\log\left(\frac{1}{\sqrt{\alpha}}\log\left(\frac{1}{\sqrt{\alpha}}\right)\right)\quad\sim\quad c_{g}c_{G\star g}\cdot\frac{2\pi^{\frac{d}{2}}}{\Gamma\left(\frac{d}{2}\right)}\cdot\frac14\log n,
\end{align*}
and by \Cref{lem:g_alpha_x_bound} we can bound the second sum as
\begin{align*}
\sum_{x:\:||x||>\frac{1}{\sqrt{\alpha}}\log\left(\frac{1}{\sqrt{\alpha}}\right)}g(x)(g\star G)(x)\quad\lesssim\quad
\sum_{x:\:||x||>\frac{1}{\sqrt{\alpha}}\log\left(\frac{1}{\sqrt{\alpha}}\right)}\frac{e^{-c||x||\sqrt{\alpha}}}{||x||^8}\quad\lesssim\quad1.
\end{align*}

Also, we can lower bound as
\begin{align*}
&\sum_{x}g_{\alpha}(x)(g\star G)(x)\quad\\
&\quad\ge \sum_{x:\:||x||\le\frac{1}{\sqrt{\alpha}}\log\left(\frac{1}{\sqrt{\alpha}}\right)^{-1}}g(x)(g\star G)(x)- \sum_{x:\:||x||\le\frac{1}{\sqrt{\alpha}}\log\left(\frac{1}{\sqrt{\alpha}}\right)^{-1}}(g(x)-g_{\alpha}(x))(g\star G)(x).
\end{align*}

Here the first sum is $\sim c_{g}c_{G\star g}\cdot\frac{2\pi^{\frac{d}{2}}}{\Gamma\left(\frac{d}{2}\right)}\cdot\frac14\log n$, while from \Cref{lem:g_alpha_x_bound} we know that the second sum is $\ll\log n$.

So overall we get that
\begin{align*}
\E{G_n^+}\quad\sim\quad\left(c_{g}\cdot c_{G\star g}\cdot\frac{\sigma^2}{2}\cdot\frac{\pi^{\frac{d}{2}}}{2\Gamma\left(\frac{d}{2}\right)}\right)\log n.
\end{align*}
An analogous calculation shows that we also have
$\E{G_n^-}\sim\left(c_{g}\cdot c_{G\star g}\cdot\frac{\sigma^2}{2}\cdot\frac{\pi^{\frac{d}{2}}}{2\Gamma\left(\frac{d}{2}\right)}\right)\log n.$

So overall we proved that
\begin{align}
\E{G_n}\quad\sim\quad\left(c_{g}\cdot c_{G\star g}\cdot\frac{\sigma^2}{2}\cdot\frac{\pi^{\frac{d}{2}}}{\Gamma\left(\frac{d}{2}\right)}\right)\log n,
\end{align}
finishing the proof of~\eqref{eq:E_Gn}, and hence the proof of \Cref{pro:E_Un}.

\subsection{Variance of $G_n^+$ and $G_n^-$}\label{sec:Var_Gn}

In this section we prove~\eqref{eq:Var_Gn}.

We can write
\begin{align*}
\Var{G_n^{+}}\quad\le\quad\sum_{k\ge0}\E{Y^{+}_n(k)^2}+2\sum_{k\ge0}\sum_{m\ge1}\Cov{Y^{+}_n(k)}{Y^{+}_n(k+m)},
\end{align*}
where $Y^{+}_n(k)=\sum_{\ell\ge1}Y^{+}_n(k,\ell)$, and $Y^{+}_n(k,\ell)$ are as in \Cref{sec:E_Gn}.

\subsubsection*{Diagonal terms in the variance of $G_n^+$}

In what follows we prove that
\begin{align}\label{eq:sum_Y_squared}
\sum_{k\ge0}\E{Y_n^+(k)^2}\quad\lesssim\quad\log n\:.
\end{align}

We can write
\begin{align}
\label{eq:E_Ysquared}
&\E{Y^{+}_n(k,\ell)^2}\quad=\quad\pr{d_k^{+}\ge\ell}\E{\til{\theta}^{d_0^{+}+...+d_{k-1}^{+}+\ell-1}}\lambda\:\cdot\\
\nonumber
&\quad\cdot\left(\E{\sum_{i=0}^{|\GW|-1}\lambda^i G(X_{k+1}+T_i)^2}+ 2\E{\sum_{0\le i<j<|\GW|}\lambda^j G(X_{k+1}+T_i)G(X_{k+1}+T_j)}\right),\\
\label{eq:E_Y_cross_term}
&\E{Y^{+}_n(k,\ell)Y^{+}_n(k,\ell+r)}\quad=\quad\pr{d_k^{+}\ge\ell+r}\E{\til{\theta}^{d_0^{+}+...+d_{k-1}^{+}+\ell+r-2}}\lambda\:\cdot\\
\nonumber
&\quad\cdot\E{\econd{\lambda^{|\GW|}\sum_{i=0}^{|\GW|-1}G(X_{k}+\til{X}_1+T_i)}{X_k}\econd{\sum_{j=0}^{|\GW|-1}\lambda^jG(X_{k}+\what{X}_1+T_j)}{X_k}}.
\end{align}

Summing the first term in~\eqref{eq:E_Ysquared} over $k\ge0$ and $\ell\ge1$, we get an expression that is\\ $\lesssim\sum_{k\ge0}\sum_{j\ge0}\theta_1^k\theta^j\E{G(X_{k+j+1})^2}$. This is then $\lesssim\sum_{x}\sum_{y}g_{1-\theta_1}(x)g_{1-\theta}(y)G(x+y)$\\ $\lesssim\sum_{y}g_{1-\theta}(y)(g\star G)(y)$. Using \Cref{cor:theta_thetatil_thetadot} we can upper bound this analogously to~\eqref{eq:E_Gplus_sum}, and get that it is $\lesssim\log n$.

Summing the second term in~\eqref{eq:E_Ysquared} over $k\ge0$ and $\ell\ge1$, we get an expression that is \[\lesssim\quad\sum_{x,y,z,w}g_{1-\theta_1}(x)g_{1-\theta}(y)g_{1-\theta}(z)g_{1-\what{\theta}}(w)G(x+y+z)G(x+y+w)\quad\lesssim\quad\log n.\]

Summing~\eqref{eq:E_Y_cross_term} over $k\ge0$, $\ell\ge1$, $r\ge1$ we get an expression that is
\begin{align*}
\lesssim\quad\sum_{x,y,z}g_{1-\theta_1}(x)g_{1-\what{\theta}}(y)g_{1-\theta}(z)G(x+y)G(x+z)\quad\lesssim\quad1.
\end{align*}

This finishes the proof of~\eqref{eq:sum_Y_squared}.

\subsubsection*{Covariance terms in the variance of $G_n^+$}

In what follows we prove that
\begin{align}\label{eq:Cov_Y}
\sum_{k\ge0}\sum_{m\ge1}\Cov{Y^{+}_n(k)}{Y^{+}_n(k+m)}\quad\lesssim\quad\log n.
\end{align}

For $k\ge1$, $m\ge1$ we have
\begin{align*}
&\E{Y_n^{+}(k)Y_n^{+}(k+m)}\quad=\quad\sum_{\ell\ge1}\sum_{r\ge1}\E{Y_n^{+}(k,\ell)Y_n^{+}(k+m,r)}\\
&=\quad\sum_{\ell\ge1}\sum_{r\ge1}\mathbb{E}\Bigg[\1{d^{+}_{k}\ge\ell}\1{d^{+}_{k+m}\ge r}\E{\lambda^{|\GW|}}^{d^{+}_0+...+d^{+}_{k+m-1}+r-2}\cdot\\
&\hspace{2cm}\cdot\econd{\lambda^{|\GW|}\sum_{i=0}^{|\GW|-1}G(X_k+\til{X}_1+T_i)}{X_k}\econd{\sum_{i=0}^{|\GW|-1}\lambda^{i}G(X_{k+m}+\what{X}_1+T_i)}{X_k}\Bigg].
\end{align*}
Using the definition of $\theta_0$ and $\theta_1$ in~\eqref{eq:theta0_def}, and using~\eqref{eq:GW_one_sum_all} and~\eqref{eq:GW_one_sum_fwd}, we get that this is
\begin{align*}
&=\quad\theta_0\theta_1^{k+m-2}\left(\sum_{\ell\ge1}\E{\1{d^{+}_k\ge\ell}\til{\theta}^{d^{+}_k}}\right)\left(\sum_{r\ge1}\pr{d^{+}_{k+m}\ge r}\til{\theta}^{r-2}\right)\cdot\\
&\hspace{6cm}\cdot\sum_{i\ge0}\sum_{j\ge0}\E{G(X_k+\til{X}_{i+1})G(X_{k+m}+\what{X}_{j+1})}\what{\theta}^i\til{\theta}\theta^j.
\end{align*}
After some further simplifications, we get that
\begin{align}\label{eq:EYY}
&\E{Y_n^{+}(k)Y_n^{+}(k+m)}\\
\nonumber
&\quad=\quad\theta_0\theta_1^{k+m-2}\E{d^+_1\til{\theta}^{d^+_1}}\frac{1-\theta_1}{1-\til{\theta}}\sum_{i\ge0}\sum_{j\ge0}\E{G(X_k+\til{X}_{i+1})G(X_{k}+\what{X}_{m+j+1})}\what{\theta}^i\theta^j.
\end{align}
Meanwhile,
\begin{align*}
&\E{Y_n^{+}(k)}\E{Y_n^{+}(k+m)}\quad=\quad\sum_{\ell\ge1}\sum_{r\ge1}\E{Y_n^{+}(k,\ell)}\E{Y_n^{+}(k+m,r)}\\
&=\left(\sum_{\ell\ge1}\pr{d^{+}_{k}\ge\ell}\E{\E{\lambda^{|\GW|}}^{d^{+}_0+...+d^{+}_{k-1}+\ell-1}}\E{\sum_{i=0}^{|\GW|-1}\lambda^{i}G(X_{k}+\til{X}_1+T_i)}\right)\cdot\\
&\hspace{0.6cm}\cdot\left(\sum_{r\ge1}\pr{d^{+}_{k+m}\ge r}\E{\E{\lambda^{|\GW|}}^{d^{+}_0+...+d^{+}_{k+m-1}+r-1}}\E{\sum_{i=0}^{|\GW|-1}\lambda^{i}G(X_{k+m}+\what{X}_1+T_i)}\right).
\end{align*}
Again, using the definition of $\theta_0$ and $\theta_1$ in~\eqref{eq:theta0_def}, and using~\eqref{eq:GW_one_sum_fwd}, we get that this is
\begin{align*}
&=\quad\theta_0^2\theta_1^{2k+m-2}\left(\sum_{\ell\ge1}\pr{d^{+}_{k}\ge \ell}\til{\theta}^{\ell-1}\right)\left(\sum_{r\ge1}\pr{d^{+}_{k+m}\ge r}\til{\theta}^{r-1}\right)\cdot\\
&\hspace{5cm}\cdot\sum_{i\ge0}\sum_{j\ge0}\E{G(X_k+\til{X}_{i+1})}\E{G(X_{k+m}+\what{X}_{j+1})}\theta^i\theta^j,
\end{align*}
and after some further simplifications we get that
\begin{align}\label{eq:EYEY}
&\E{Y_n^{+}(k)}\E{Y_n^{+}(k+m)}\\
\nonumber
&\quad=\quad\theta_0^2\theta_1^{2k+m-2}\left(\frac{1-\theta_1}{1-\til{\theta}}\right)^2\sum_{i\ge0}\sum_{j\ge0}\E{G(X_k+\til{X}_{i+1})}\E{G(X_{k}+\what{X}_{m+j+1})}\theta^i\theta^j.
\end{align}
Taking the difference between~\eqref{eq:EYY} and~\eqref{eq:EYEY}, and using that $\theta_0,\theta_1=1-O\left(\frac{1}{\sqrt{n}}\right)$ (by \Cref{lem:theta0_theta_1}), $\frac{1-\theta_1}{1-\til{\theta}}\asymp1$ (by \Cref{lem:theta_fraction}), $\theta_1^k=1-O\left(\frac{k}{\sqrt{n}}\right)$, $\E{d^+_1\til{\theta}^{d^+_1}}\le\frac{1-\theta_1}{1-\til{\theta}}$, and $\what{\theta}\le\theta$, we get that
\begin{align}\nonumber
&\Cov{Y_n^{+}(k)}{Y_n^{+}(k+m)}\\
\nonumber
&\lesssim\quad\theta_1^{k+m}\sum_{i\ge0}\sum_{j\ge0}\theta^i\theta^j\Big(\E{G(X_k+\til{X}_{i+1})G(X_{k}+\what{X}_{m+j+1})}\\
\nonumber
&\hspace{5cm}-\E{G(X_k+\til{X}_{i+1})}\E{G(X_{k}+\what{X}_{m+j+1})}\Big)\\
\nonumber
&+\quad\frac{1}{\sqrt{n}}\cdot\E{Y_n^{+}(k)}\E{Y_n^{+}(k+m)}\\
\label{eq:Cov_Y_upper_bound}
&+\quad\frac{k}{\sqrt{n}}\cdot\theta_1^{k+m}\sum_{i\ge0}\sum_{j\ge0}\theta^i\theta^j\E{G(X_k+\til{X}_{i+1})}\E{G(X_{k}+\what{X}_{m+j+1})}.
\end{align}

We know that $\sum_{k,m\ge1}\E{Y_n^{+}(k)}\E{Y_n^{+}(k+m)}\le\E{G_n^+}^2\lesssim(\log n)^2$, so the contribution from the second term in~\eqref{eq:Cov_Y_upper_bound} is $\lesssim\log n$ as required.

We also have
\begin{align*}
&\sum_{k,m\ge1}\sum_{i,j\ge0}k\theta_1^{k}\theta_1^{m}\theta^i\theta^j\E{G(X_k+\til{X}_{i+1})}\E{G(X_{k}+\what{X}_{m+j+1})}\\
&\le\quad\sum_{k\ge1}k\theta_1^{k}\E{(G\star g_{1-\theta})(X_k)}\E{(G\star g_{1-\theta}\star g_{1-\theta_1})(X_k)}.
\end{align*}

Using that $(G\star g_{1-\theta})(x)\le(G\star g)(x)\asymp\frac{1}{||x||^2}$ and $(G\star g_{1-\theta}\star g_{1-\theta_1})(y)\lesssim\log n$, using \Cref{lem:galpha_weighted}, and the bound $\frac{1}{||x||^{12}+||y||^{12}}\lesssim\frac{1}{||x||^5}\frac{1}{||y||^{7}}$, we get that this is
\begin{align*}
&\lesssim\quad
\left(\sum_{x:\:||x||\le Cn^{\frac14}}\frac{1}{||x||^5}\frac{1}{||x||^2}+ \sum_{x:\:||x||\ge Cn^{\frac14}}e^{-c||x||n^{-\frac14}}\frac{1}{||x||^5}\frac{1}{||x||^2}\right)\\
&\hspace{4cm}\cdot\left(\sum_{y:\:||y||\le Cn^{\frac14}}\frac{1}{||y||^7}\log n
+\sum_{y:\:||y||\ge Cn^{\frac14}}e^{-c||y||n^{-\frac14}}\frac{1}{||y||^7}\log n\right).
\end{align*}
This is $\asymp\sqrt{n}\log n$, so the overall contribution from the third term in~\eqref{eq:Cov_Y_upper_bound} is also $\lesssim\log n$.

Finally, note that the first sum in~\eqref{eq:Cov_Y_upper_bound} is
\begin{align*}
&\le\quad\theta^{-2}\theta_1^{k+m}\Cov{(G\star g_{1-\theta})(X_k)}{(G\star g_{1-\theta})(X_k+\what{X}_m)}\\
&\qquad+\:\sum_{j\ge0}\theta^j\E{G(X_k)}\E{G(X_k+\what{X}_{m+j})}\:+ \:\sum_{i\ge0}\theta^i\E{G(X_k+\til{X}_{i})}\E{G(X_k+\what{X}_{m})}.
\end{align*}
Summing the second and third term over $k\ge1$, $m\ge1$, we get an expression that is $\lesssim\log n$.

Summing the first term over $m\ge0$, we get $\theta^{-2}\theta_1^{k}\Cov{(G\star g_{1-\theta})(X_k)}{(G\star g_{1-\theta}\star g_{1-\theta_1})(X_k)}$. The contribution from the $m=0$ case is $\lesssim\log n$, and we have $\theta^{-2}\asymp1$. We conclude bounding the contribution from the first term in~\eqref{eq:Cov_Y_upper_bound} by using the following result.

\begin{lemma}\label{lem:Cov_functions_of_Xk}
We have
\begin{align}\label{eq:sum_Cov_g_conv}
\sum_{k\ge1}\theta_1^k\cdot\Cov{(G\star g_{1-\theta})(X_k)}{(G\star g_{1-\theta}\star g_{1-\theta_1})(X_k)}\quad\lesssim\quad\log n.
\end{align}
\end{lemma}

\begin{proof}[Sketch proof.]
Firstly, we use \Cref{lem:g_alpha_x_bound} to compare $(G\star g_{1-\theta})(x)$ and $(G\star g_{1-\theta}\star g_{1-\theta_1})(x)$ to $a(x):=\frac{c_a}{||x||^2}\1{||x||\le n^{\frac14}}$ and\\ $b(x):=c_b\left(\log\left(2n^{\frac14}-||x||\right)+\log\left(2n^{\frac14}+||x||\right)-2\log\left(||x||\right)\right)\1{||x||\le n^{\frac14}}$, respectively,\\ where $c_a$ and $c_b$ are positive constants. We show that the contribution to the sum in~\eqref{eq:sum_Cov_g_conv} from the error terms is $\lesssim\log n$.

Then we compare $\sum_{k\ge1}\theta_1^k\cdot\Cov{a(X_k)}{b(X_k)}$ to \begin{align}\label{eq:sum_Cov_g_conv_approx}
\sum_{x:\:||x||\le n^{\frac14}}a(x)\sum_{y:\:||y||\le n^{\frac14}}g^{(2d)}_{1-\theta_1}\left((x,y)\right)\left(b(x)-b(y)\right),
\end{align}
where $g^{(2d)}_{1-\theta_1}$ denotes the Green's function for a SRW on $\Z^{2d}$ induced by the pairs in $\Z^d\times\Z^d$ with an even sum of coordinates. Note that $g^{(2d)}_{1-\theta_1}\left((x,y)\right)\lesssim\frac{1}{\left(||x||^2+||y||^2\right)^7}$, which is $\lesssim\frac{1}{||x||^k}\frac{1}{||y||^{14-k}}$ for any $k\in\{0,1,....,14\}$.

We split the sum~\eqref{eq:sum_Cov_g_conv_approx} based on the three different log terms in $b$, and for each of them we split based on whether the contribution from $x$ or $y$ is larger. Finally, for each of the six sums we upper bound its absolute value by using $g^{(2d)}_{1-\theta_1}\left((x,y)\right)\lesssim\frac{1}{||x||^k}\frac{1}{||y||^{14-k}}$ for a suitable choice of $k$.
\end{proof}

Finally, we use $\eqref{eq:GW_one_sum_all}$ and bounds analogous to the ones in \Cref{sec:E_Gn} to show that
\[\sum_{m\ge1}\Cov{Y_n^+(0)}{Y_n^+(m)}\quad\le\quad\sum_{m\ge1}\E{Y_n^+(0)Y_n^+(m)}\quad\lesssim\quad\log n.\]

This finishes the proof of~\eqref{eq:Cov_Y}.

\subsection*{Conclusion}

Combining~\eqref{eq:E_Ysquared} and~\eqref{eq:Cov_Y}, we get that $\Var{G_n^+}\lesssim\log n$. The diagonal terms and covariance terms in $\Var{G_n^-}$ are handled analogously, so we also get $\Var{G_n^-}\lesssim\log n$. Then by Cauchy-Schwarz we have $\Var{\econd{U_n}{\T}}=\Var{G_n}\lesssim\log n$. This finishes the proof of~\eqref{eq:Var_Gn}.

\subsection{Bounding $\E{\varcond{U_n}{\T}}$}

In this section we sketch the proof of~\eqref{eq:E_Varcond_Un}.

Let us write $G_{\T}(x)=\sum_{i=-\xiln}^{\xirn}\1{\T_i=x}$, $G_{\til{\T}}(x)=\sum_{i=0}^{\infty}\1{\til{\T}_i=x}$. Then we can write
\begin{align*}
\E{\varcond{U_n}{\T}}\quad&\le\quad\E{\varcond{\Ell_n}{\T}}\quad=\quad\E{\varcond{\sum_{x}G_{\T}(x)G_{\til{\T}}(x)}{\T}}\\
&=\quad\E{\sum_{x,y}\covcond{G_{\T}(x)G_{\til{\T}}(x)}{G_{\T}(y)G_{\til{\T}}(y)}{\T}}\\
&=\quad\E{\sum_{x,y}G_{\T}(x)G_{\T}(y)\Cov{G_{\til{\T}}(x)}{G_{\til{\T}}(y)}}\\
&=\quad\sum_{x,y}\E{G_{\T}(x)G_{\T}(y)}\Cov{G_{\til{\T}}(x)}{G_{\til{\T}}(y)}.
\end{align*}

Then for each $x$, $y$ we can bound $\E{G_{\T}(x)G_{\T}(y)}$ and $\Cov{G_{\til{\T}}(x)}{G_{\til{\T}}(y)}$ by estimates similar to the ones in \Cref{sec:Var_Gn}, but using functions $\delta_x$ and $\delta_y$ instead of $G$. The details of these calculations are omitted.

This concludes the proof of~\eqref{eq:E_Varcond_Un}. Combining~\eqref{eq:Var_Gn} and~\eqref{eq:E_Varcond_Un}, we can conclude \Cref{pro:Var_Un}.

\section{Proof of Theorems~\ref{thm:prob_nocap_nozero} and~\ref{thm:prob_nocap_nozero_n}}\label{sec:prob_nocap_nozero}

In this section we prove \Cref{thm:prob_nocap_nozero}, i.e.\ we show that $\E{\1{\A_n}\1{\B_n}}\sim\frac{c_8}{\log n}$, where $\A_n$ and $\B_n$ are defined as in \Cref{sec:shift_invariance_magic_formula}.

Once we have this, the proof of \Cref{thm:prob_nocap_nozero_n} is very quick.

\begin{proof}[Proof of \Cref{thm:prob_nocap_nozero_n}]
Note that $\pr{\til{\T}(-\infty,0)\cap\T[-n,n]=\emptyset,\:0\not\in\T(0,n]}\ge$\\  $\pr{\til{\T}(-\infty,0)\cap\T[-\xi^{\ell}_{n(\log n)^2},\xi^{r}_{n(\log n)^2}]=\emptyset,\: 0\not\in\T(0,\xi^{r}_{n(\log n)^2}]}-2\pr{\xi^{\ell}_{n(\log n)^2}<n}$, then use  \Cref{thm:prob_nocap_nozero} with $n(\log n)^2$ instead of $n$ and use the concentration of $\xi^{\ell}_{n(\log n)^2}$.

The upper bound is analogous, using \Cref{thm:prob_nocap_nozero} with $\frac{n}{(\log n)^2}$.
\end{proof}

Working towards the proof of \Cref{thm:prob_nocap_nozero}, we will prove the following two results.

\begin{lemma}\label{lem:prob_AnBn_upper_bound}
We have
\begin{align*}
\E{\1{\A_n}\1{\B_n}}\quad\lesssim\quad\frac{1}{\log n}.
\end{align*}
\end{lemma}

\begin{proposition}\label{pro:limsup_prob_ABD}
Let $\eps$ be any positive constant and let
\begin{align*}
D_{\eps}:=\{\left|U_n-\E{U_n}\right|>\eps\E{U_n}\}.    
\end{align*}
Then we have
\begin{align*}
\limsup\limits_{n\to\infty}\:(\log n)\pr{\A_n,\B_n,D_{\eps}}=0.
\end{align*}
\end{proposition}

The proof of \Cref{lem:prob_AnBn_upper_bound} is analogous to the proof of the first part of~\cite[Lemma 3.2]{BCap_of_RW_range}, and is presented here for completeness.

\begin{proof}[Proof of \Cref{lem:prob_AnBn_upper_bound}]
Let $\Ecal_n=\left\{\left|U_n-\E{U_n}\right|\ge\frac12\E{U_n}\right\}$. Then we can write
\begin{align*}
\E{\1{\A_n}\1{\B_n}}\quad\le\quad\pr{\Ecal_n}+\quad\E{\1{\A_n}\1{\B_n}\1{\Ecal_n^c}}.
\end{align*}
Using \Cref{pro:E_Un} and \Cref{pro:Var_Un} and Chebyshev's inequality we get that $\pr{\Ecal_n}\lesssim\frac{1}{\log n}$.

Note that on $\Ecal_n^c$ we have $U_n\ge\frac12\E{U_n}$, hence $\E{\1{\A_n}\1{\B_n}\1{\Ecal_n^c}}\le\frac{2\E{\1{\A_n}\1{\B_n}U_n}}{\E{U_n}}$. By \Cref{cor:magic_formula} and \Cref{pro:E_Un} this is also $\lesssim\frac{1}{\log n}$.
\end{proof}

Once we have \Cref{lem:prob_AnBn_upper_bound} and \Cref{pro:limsup_prob_ABD}, the proof of \Cref{thm:prob_nocap_nozero} can be completed as follows. This is similar to the proof of~\cite[Proposition 3.5]{BCap_of_RW_range}.

\begin{proof}[Proof of \Cref{thm:prob_nocap_nozero}]
By \Cref{cor:magic_formula} we can write
\[
\E{\1{\A_n}\1{\B_n}}\quad=\quad\frac{1}{\E{U_n}}-\frac{\E{\1{\A_n}\1{\B_n}\left(U_n-\E{U_n}\right)}}{\E{U_n}}.
\]
From \Cref{pro:E_Un} we know that $\frac{1}{\E{U_n}}\sim\frac{c_8}{\log n}$. We can bound the second term as follows.

Note that
\begin{align*}
\E{\1{\A_n}\1{\B_n}\left|U_n-\E{U_n}\right|}\quad&\le\quad \eps\E{\1{\A_n}\1{\B_n}}\E{U_n}+\E{\1{\A_n}\1{\B_n}\1{D_{\eps}}\left|U_n-\E{U_n}\right|}.
\end{align*}
By \Cref{lem:prob_AnBn_upper_bound} and \Cref{pro:E_Un} we know that $\E{\1{\A_n}\1{\B_n}}\lesssim\frac{1}{\log n}$ and $\E{U_n}\asymp\log n$, hence $\limsup_{n\to\infty}\eps\E{\1{\A_n}\1{\B_n}}\E{U_n}\to0$ as $\eps\to 0$. Also,
\begin{align*}
\E{\1{\A_n}\1{\B_n}\1{D_{\eps}}\left|U_n-\E{U_n}\right|}\:&\le\: \pr{\A_n,\B_n,D_{\eps}}^{\frac12}\Var{U_n}^{\frac12}\:\lesssim\:\pr{\A_n,\B_n,D_{\eps}}^{\frac12}(\log n)^{\frac12},
\end{align*}
and by \Cref{pro:limsup_prob_ABD} we know that $\limsup\limits_{n\to\infty}\:(\log n)\pr{\A_n,\B_n,D_{\eps}}=0$ for any $\eps$.

Combining these, we get that $\E{\1{\A_n}\1{\B_n}\left|U_n-\E{U_n}\right|}\ll1$, hence $\E{\1{\A_n}\1{\B_n}}\sim\frac{1}{\E{U_n}}\sim\frac{c_8}{\log n}$ as required.
\end{proof}

The rest of this section is devoted to the proof of \Cref{pro:limsup_prob_ABD}.

\subsection{Splitting the event $\A_n\cap\B_n\cap D_{\eps}$ into four cases}

In this section we upper bound $\pr{\A_n\cap\B_n\cap D_{\eps}}$ by probabilities of events that only depend on certain parts of the trees $\til{\T}$ and $\T$.

For ease of notation we will often drop the subscript $n$, and unless specified otherwise, we write $\T$ to denote $\T\left[-\xiln,\xirn\right]$. We also write
\begin{align*}
\T_{-}:=\T[-\xiln,0),\qquad\T_{+}:=\T(0,\xirn],\qquad\til{\T}_{-}:=\til{\T}(-\infty,0),\qquad\til{\T}_{+}:=\til{\T}(0,\infty).
\end{align*}
For a given constant $\til{C}$, we write $\til{\T}(\le\til{C})$ to denote the part of $\til{\T}$ until the random walk $\Xtil$ on its spine first hits $\partial B(0,\til{C})$ (i.e.\ we keep the spine vertices until the first one where the label is at distance $\ge\til{C}$ from $0$ in $\Z^d$, and also keep the trees hanging off these spine vertices). Similarly, we write $\til{\T}(\ge\til{C})$ to denote the part of $\til{\T}$ after the random walk $\Xtil$ on its spine first hits $\partial B(0,\til{C})$. We define $\T(\le C)$, $\T(\ge C)$, $\til{\T}_{+}(\le\til{C})$, etc analogously. (Recall that $\T$ stands for $\T\left[-\xiln,\xirn\right]$, so $\T(\ge C)$ might be empty in case $\xiln$ and $\xirn$ are small.)

In what follows $\til{C}$, $C_1$ and $C_2$ will be constants, with values to be specified later. We let
\begin{align*}
U_n(\vee,\wedge):=\econd{\#\text{ intersections of }\til{\T}_+(\le\til{C})\text{ and }\T(>C_1)}{\T,\til{X}, (\til{d}_i^{\pm})},
\end{align*}
and we define $U_n(\vee,\vee)$ analogously, with $\til{\T}_+(\le\til{C})$ and $\T(\le C_1)$. We also define $U_n(\wedge,\wedge)$ and $U_n(\wedge,\vee)$ analogously, but with $C_2$ instead of $C_1$. \footnote{In this notation $\vee$ refers to the part of the tree near the root, $\wedge$ refers to the part further from the root, the first argument refers to $\til{\T}$, and the second argument refers to $\T$.}

Note that $U_n=\sum_{a,b\in\{\wedge,\vee\}}U_n(a,b)$, hence if $D_{\eps}$ holds, then\\ $D_{\eps}(a,b):=\left\{\left|U_n(a,b)-\E{U_n(a,b)}\right|>\frac14\eps\E{U_n}\right\}$ holds for some $a,b\in\{\wedge,\vee\}$.

To facilitate some of the later calculations, we will also consider the following events. Let us say that a set $\Gamma$ does not disconnect a point $x\in\partial B(0,r)$ from $0$ in $B(0,r)$ if there exists a time $t$ and a path $\gamma[0,t]$ from $0$ to $x$ in $B(0,r)$ such that $\gamma(0,t]\cap\Gamma=\emptyset$. Note that if $\til{\T}_{-}\cap\T=\emptyset$ then the following events all hold.
\begin{align*}
A(\wedge,\wedge):=\quad\Big\{&\T\left(\ge4C_1\right)\text{ does not disconnect }\Xtil_{\til{\tau}_{4\til{C}}}\text{ from }0\text{ in }B(0,4\til{C}),\\
&\til{\T}\left(\ge4\til{C}\right)\text{ does not disconnect }X_{\tau_{4C_1}}\text{ from }0\text{ in }B(0,4C_1)\Big\},\\
A(\wedge,\vee):=\quad\Bigg\{&\T\left(\le\frac14C_1\right)\text{ does not disconnect }\Xtil_{\til{\tau}_{4\til{C}}}\text{ from }0\text{ in }B\left(0,4\til{C}\right)\Bigg\},\\
A(\vee,\wedge):=\quad\Big\{&\til{\T}_{-}\left(\le\frac14\til{C}\right)\text{ does not disconnect }0\text{ and }X_{\tau_{4C_2}}\text{ in }B\left(0,4C_2\right)\Big\},\\
A(\vee,\vee):=\quad\Big\{&\text{an event that always holds, defined for notational convenience}\Big\}.
\end{align*}

So we can write
\begin{align}
&\pr{\A_n,\B_n,D_{\eps}}\quad\le\quad\sum_{a,b\in\{\wedge,\vee\}}\notag \pr{\til{\T}_{-}\cap\T=\emptyset,\:0\not\in\T_{+},\:D_{\eps}(a,b)}\\
&\label{eq:vv} \le\quad
\pr{\til{\T}_{-}\left(\ge4\til{C}\right)\cap\T\left(\ge4C_1\right)=\emptyset,\:0\not\in\T_{+}\left(\ge4C_1\right),\:A(\wedge,\wedge),\:D_{\eps}(\vee,\vee)}\\
&\label{eq:va} +\quad\pr{\til{\T}_{-}\left(\ge4\til{C}\right)\cap\T\left(\le\frac14C_1\right)=\emptyset,\:0\not\in\T_{+}\left(\le\frac14C_1\right),\:A(\wedge,\vee),\:D_{\eps}(\vee,\wedge)}\\
&\label{eq:av} +\quad\pr{\til{\T}_{-}\left(\le\frac14\til{C}\right)\cap\T\left(\ge4C_2\right)=\emptyset,\:0\not\in\T_{+}(\ge4C_2),\:A(\vee,\wedge),\:D_{\eps}(\wedge,\vee)}\\
&\label{eq:aa} +\quad\pr{\til{\T}_{-}\left(\le\frac14\til{C}\right)\cap\T\left(\le\frac14C_2\right)=\emptyset,\:0\not\in\T_{+}\left(\le\frac14C_2\right),\:A(\vee,\vee),\:D_{\eps}(\wedge,\wedge)}.
\end{align}

\subsection{Pulling out the $D_\eps$ terms and bounding their probabilities}\label{sec:decorrelate_Deps}

In this section we explain how to decorrelate the $D_{\eps}(a,b)$ terms from the rest of the events in the probabilities~\eqref{eq:vv}-\eqref{eq:aa}, and how to bound the probabilities $\pr{D_{\eps}(a,b)}$. More precisely, we show the following.
\begin{lemma}\label{lem:decorrelate_D}
We have
\begin{align*}
\eqref{eq:vv}\:\lesssim\:\pr{\til{\T}_{-}\left(\ge4\til{C}\right)\cap\T\left(\ge4C_1\right)=\emptyset,\:0\not\in\T_{+}\left(\ge4C_1\right),\:A(\wedge,\wedge)}\pr{D_{\eps}(\vee,\vee)} + o\left(\frac{1}{\log n}\right),
\end{align*}
where the implicit constant in the $\lesssim$ does not depend on $\til{C}$, $C_1$ and $C_2$.

We also have analogous bounds for~\eqref{eq:va},~\eqref{eq:av} and~\eqref{eq:aa}.
\end{lemma}

\begin{lemma}\label{lem:bound_Dab_prob}
For every $a,b\in\{\wedge,\vee\}$ we have
\begin{align*}
\limsup_{n\to\infty}\:(\log n)\:\pr{D_{\eps}(a,b)}\quad\lesssim\quad\frac{1}{\eps^2}\:,
\end{align*}
where the implicit constants in the $\lesssim$'s do not depend on $\til{C}$, $C_1$ and $C_2$.
\end{lemma}

\begin{proof}[Proof of \Cref{lem:decorrelate_D}]
We will focus on~\eqref{eq:vv}, as the proofs for the other three probabilities work analogously.

Note that the event $D_{\eps}(\vee,\vee)$ depends on $\left(\til{\T}_{-}\left(<\til{C}\right),\T\left(<C_1\right)\right)$, while the rest of the events in probability~\eqref{eq:vv} depend on $\left(\til{\T}_{-}\left(\ge4\til{C}\right),\T\left(\ge4C_1\right)\right)$. The only dependence between these different parts of the trees is via the points $\til{X}_{\til{\tau}_{\til{C}}}$, $\til{X}_{\til{\tau}_{4\til{C}}}$, $X_{\tau_{C_1}}$, $X_{\tau_{4C_1}}$ where the spines hit given distances from 0, and via $\xiln$ and $\xirn$. We will use \Cref{lem:Harnack} to remove the dependence via the hitting points on the spines, and we will remove the dependence via $\xiln$, $\xirn$ by noting that $\xiln$, $\xirn$ are very unlikely to terminate $\T$ before its spine reaches distance $4C_1$ from 0.

For $x\in\Z^d$, $r>0$ let us write
\begin{align*}
&\til{\Omega}_{r,x}:=\left\{\til{X}_{\til{\tau}_{r}}=x\right\},\quad
\Omega_{r,x}:=\left\{X_{\tau_{r}}=x\right\},\quad\\
&\Omega_r:=\left\{\T\left[-\xiln,0\right)\supsetneq\T_{-}(\le r),\: \T\left(0,\xirn\right]\supsetneq\T_{+}(\le r)\right\}.
\end{align*}
Note that for any $x,y\in\Z^d$, conditional on the event $\til{\Omega}_{\til{C},x}\cap\til{\Omega}_{4\til{C},y}$, the parts $\til{\T}\left(<\til{C}\right)$ and $\til{\T}\left(\ge4\til{C}\right)$ are independent. Also, conditional on the event $\Omega_{C_1,x}\cap\Omega_{4C_1,y}\cap\Omega_{4C_1}$, the parts $\T\left(<C_1\right)$ and $\T\left(\ge4C_1\right)$ are independent.

Also, note that
\begin{align*}
1-\pr{\Omega_{4C_1}}\:\le\:2\left(1-\E{\til{\theta}^{d_0^++...+d_{\tau_{4C_1}}^++\tau_{4C_1}}}\right)\:\lesssim\:
\pr{\tau_{4C_1}>n^{\frac14}}+\left(1-\theta_1^{n^{\frac14}}\til{\theta}^{n^{\frac14}}\right),
\end{align*}
and by \Cref{lem:hitting_time_tail_bounds} and \Cref{cor:theta_thetatil_thetadot} this is $\ll\frac{1}{\log n}$ as $n\to\infty$.

Therefore we have
\begin{align*}
\eqref{eq:vv}\le&\sum_{x,y,\til{x},\til{y}}
\pr{\til{\Omega}_{\til{C},\til{x}}}\prcond{\til{\Omega}_{4\til{C},\til{y}}}{\til{\Omega}_{\til{C},\til{x}}}\pr{\Omega_{C_1,x}}\prcond{\Omega_{4C_1,y}}{\Omega_{C_1,x}}\prcond{\Omega_{4C_1}}{\Omega_{C_1,x},\Omega_{4C_1,y}}\\
&\cdot\prcond{\til{\T}_{-}\left(\ge4\til{C}\right)\cap{\T}\left(\ge4C_1\right)=\emptyset,\:0\not\in{\T}_{+}\left(\ge4C_1\right),\:{A}(\wedge,\wedge)}{\til{\Omega}_{4\til{C},\til{y}},\:\Omega_{4C_1,y},\:\Omega_{4C_1}}\\
&\cdot\prcond{{D}_{\eps}(\vee,\vee)}{\til{\Omega}_{\til{C},\til{x}},\:\Omega_{C_1,x},\:\Omega_{C_1}} \quad+o\left(\frac{1}{\log n}\right).
\end{align*}
By \Cref{lem:Harnack} we have
\begin{align*}
\prcond{\til{\Omega}_{4\til{C},\til{y}}}{\til{\Omega}_{\til{C},\til{x}}}\prcond{\Omega_{4C_1,y}}{\Omega_{C_1,x}}\quad
\lesssim\quad\pr{\til{\Omega}_{4\til{C},\til{y}}}\pr{\Omega_{4C_1,y}}.
\end{align*}
Also, note that for any $x$ with $\pr{\Omega_{C_1,x}}>0$ we have 
$\pr{\Omega_{C_1,x}}\asymp_{C_1}1$, hence $1-\prcond{\Omega_{C_1}}{\Omega_{C_1,x}}\le\frac{1-\pr{\Omega_{C_1}}}{\pr{\Omega_{C_1,x}}}\ll1$. Similarly, $1-\prcond{\Omega_{4C_1}}{\Omega_{4C_1,y}}\ll1$, hence we have
\begin{align*}
\prcond{\Omega_{4C_1}}{\Omega_{C_1,x},\Omega_{4C_1,y}}\quad\le\quad1\quad\lesssim\quad \prcond{\Omega_{C_1}}{\Omega_{C_1,x}}\prcond{\Omega_{4C_1}}{\Omega_{4C_1,y}}.
\end{align*}
Using these bounds and summing over $x$, $y$, $\til{x}$ and $\til{y}$, we get that~\eqref{eq:vv} is
\begin{align*}
&\lesssim\:\prcond{\til{\T}_{-}\left(\ge4\til{C}\right)\cap{\T}\left(\ge4C_1\right)=\emptyset,\:0\not\in{\T}_{+}\left(\ge4C_1\right),\:{A}(\wedge,\wedge)}{\Omega_{4C_1}}\prcond{{D}_{\eps}(\vee,\vee)}{\Omega_{C_1}}\\
&\hspace{12.3cm} +o\left(\frac{1}{\log n}\right)\\
&=\:\pr{\til{\T}_{-}\left(\ge4\til{C}\right)\cap{\T}\left(\ge4C_1\right)=\emptyset,\:0\not\in{\T}_{+}\left(\ge4C_1\right),\:{A}(\wedge,\wedge)}\pr{{D}_{\eps}(\vee,\vee)} +o\left(\frac{1}{\log n}\right),
\end{align*}
as required.
\end{proof}

\begin{proof}[Proof of \Cref{lem:bound_Dab_prob}]
It is sufficient to prove that $\limsup_{n\to\infty}\frac{\Var{U_n(a,b)}}{\log n}\lesssim1$ for each $(a,b)$. Then using $\E{U_n}\asymp\log n$ and the Chebyshev inequality we get the result.

Using very crude bounds we see that for any $C$ we have
\begin{align*}
&\Var{G\left(\T(\le C)\right)}\quad\le\quad \E{G\left(\T(\le C)\right)^2}\quad\\
&\lesssim\sum_{x,y\in B(0,C)}g(x)g(x,y)\E{G(x+T)}\E{G(y+T)}+ \sum_{x\in B(0,C)}g(x)\E{G(x+T)^2}\:\lesssim_C\:1.
\end{align*}
Here $T$ denotes a simple random walk indexed by a Galton-Watson tree and we used \eqref{eq:GW_two_sums_fwd} as $n\to\infty$ to estimate $\E{\#(x\text{ in }T)^2}$.

The above bound also implies $\Var{G\left(\T(\ge C)\right)}\lesssim\Var{G\left(\T\right)}+\Var{G\left(\T(\le C)\right)}\lesssim\log n+O_C(1)$, hence $\Var{\econd{U_n(a,b)}{\T}}\lesssim\log n+O_{C_1,C_2}(1)$ for any $a,b\in\{\wedge,\vee\}$.

Now we explain how to bound $\E{\varcond{U_n(a,b)}{\T}}$. A quick calculation shows that
\begin{align*}
&\E{G_{\til{\T}(\le C)}(x)G_{\til{\T}(\le C)}(y)}\quad\lesssim_{ C}\quad \E{\#(x\text{ in }T)\#(y\text{ in }T)}+\E{\#(x\text{ in }T)}\E{\#(y\text{ in }T)}\\
&\hspace{2cm}\lesssim\quad\sum_{z}g(z)g(z-x)g(z-y)+g(x)g(y-x)+g(y)g(y-x)+g(x)g(y),\\
&\E{G_{\til{\T}}(x)G_{\til{\T}}(y)}\quad\lesssim\quad\sum_{z}g(z)g(z-x)(g\star g)(z-y)+\sum_{z}g(z)(g\star g)(z-x)g(z-y)\\
&\hspace{2cm}+\sum_{z}(g\star g)(z)g(z-x)g(z-y)+(g\star g)(x)g(x-y)+(g\star g)(y)g(x-y).
\end{align*}
Using these, we find that for any $C$ we have
\begin{align*}
\sum_{x,y}\E{G_{\til{\T}}(x)G_{\til{\T}}(y)}\E{G_{\til{\T}(\le C)}(x)G_{\til{\T}(\le C)}(y)}\quad\lesssim_{C}\quad1.
\end{align*}
For each of $(\vee,\vee)$, $(\vee,\wedge)$ and $(\wedge,\vee)$, the expectation $\E{\varcond{U_n(a,b)}{\T}}$ can be upper bounded by a sum of the above form. Using this and that $\E{\varcond{U_n}{\T}}\lesssim\log n$, we get that $\E{\varcond{U_n(a,b)}{\T}}$ $\lesssim\log n +O_{C_1,C_2,\til{C}}(1)$ for all $a,b\in\{\wedge,\vee\}$.

Overall these show that $\Var{U_n(a,b)}\lesssim\log n +O_{C_1,C_2,\til{C}}(1)$ for all $a,b\in\{\wedge,\vee\}$, which finishes the proof.
\end{proof}

\subsection{Bounding the non-intersection probabilities}

In this section we explain how to bound the probabilities of the non-intersection events in \Cref{lem:decorrelate_D}. We prove the following bounds.

\begin{lemma}\label{lem:Ttillarge_probs_bound}
There exists a positive constant $a(\til{C},C_1)$ depending on $\til{C}$, $C_1$ and $\sigma^2$, and a positive constant $b(\til{C})$ depending on $\til{C}$ and $\sigma^2$ such that we have
\begin{align}
\label{eq:Ttillarge_Tlarge}
\pr{\til{\T}_{-}\left(\ge4\til{C}\right)\cap\T\left(\ge4C_1\right)=\emptyset,\:0\not\in\T_{+}\left(\ge4C_1\right),\:A(\wedge,\wedge)}\quad&\le\quad\frac{a(\til{C},C_1)}{\log n},\\
\label{eq:Ttillarge_Tsmall}
\limsup_{n\to\infty}\:\pr{\til{\T}_{-}\left(\ge4\til{C}\right)\cap\T\left(\le\frac14C_1\right)=\emptyset,\:0\not\in\T_{+}\left(\le\frac14C_1\right),\:A(\wedge,\vee)}\quad&\le\quad\frac{b(\til{C})}{\log C_1}.
\end{align}
\end{lemma}

\begin{lemma}\label{lem:Ttilsmall_probs_bound}
For all $\alpha>0$ there exist $C_0,\til{C}_0>0$ such that for all $C_2\ge C_0$ and $\til{C}\ge\til{C}_0$ we have
\begin{align}\label{eq:Ttilsmall_Tsmall}
\limsup_{n\to\infty}\:\pr{\til{\T}_{-}\left(\le\frac14\til{C}\right)\cap\T\left(\le \frac14C_2\right)=\emptyset,\:0\not\in\T_{+}\left(\le \frac14C_2\right),\:A(\vee,\vee)}\quad\le\quad\alpha.
\end{align}

For all $\alpha>0$, $C_0, \til{C}_0>0$ there exist $C_2\ge C_0$, $\til{C}\ge\til{C}_0$ such that
\begin{align}\label{eq:Ttilsmall_Tlarge}
\limsup_{n\to\infty}\:\pr{\til{\T}_{-}\left(\le\frac14\til{C}\right)\cap\T\left(\ge 4C_2\right)=\emptyset,\:0\not\in\T_{+}\left(\ge 4C_2\right),\:A(\vee,\wedge)}\quad\le\quad\alpha.
\end{align}
\end{lemma}

\begin{proof}[Proof of \Cref{lem:Ttillarge_probs_bound}]
Note that for any $x\in\partial B(0,4\til{C})$ and any possible realisation $\Gamma$ of $\T$ such that there exists a path from $0$ to $x$ in $B(0,4\til{C})$ avoiding $\Gamma$, we have
\begin{align*}
\pr{\Xtil_{\tau_{4\til{C}}}=x}\quad\asymp_{\til{C}}\quad1\quad\asymp_{\til{C}}\quad\pr{\Xtil_{\tau_{4\til{C}}}=x,\:\til{\T}_{-}\left(\le4\til{C}\right)\cap\Gamma=\emptyset}.
\end{align*}

(We can see the second $\asymp_{\til{C}}$ by letting $\Xtil$ be a specific path from 0 to $x$ in $B\left(0,4\til{C}\right)$ avoiding $\Gamma$, and insisting that the spine vertices of $\til{\T}$ up to $\til{\tau}_{4\til{C}}$ have no offsprings on the past side.)

Then we also have
\begin{align*}
\pr{\Xtil_{\til{\tau}_{4\til{C}}}=x,\:\til{\T}_{-}\cap\Gamma=\emptyset}\quad\asymp_{\til{C}}\quad\pr{\Xtil_{\til{\tau}_{4\til{C}}}=x,\:\til{\T}_{-}\left(\ge4\til{C}\right)\cap\Gamma=\emptyset}.
\end{align*}

Using also that $\pr{\til{\T}_{-}\cap\T=\emptyset,\:0\not\in\T_{+}}\lesssim\frac{1}{\log n}$, we get that
\begin{align}\label{eq:Ttillarge_T}
\pr{\til{\T}_{-}\left(\ge4\til{C}\right)\cap\T=\emptyset,\:0\not\in\T_{+},\:\T\text{ does not disconnect }\Xtil_{\til{\tau}_{4\til{C}}}\text{ from }0\text{ in }B\left(0,4\til{C}\right)}\:\lesssim_{\til{C}}\:\frac{1}{\log n}.
\end{align}

Similarly, by considering a given $x\in\partial B(0,4C_1)$ and a realisation $\til{\Gamma}$ of $\til{\T}_{-}(\ge4\til{C})$ not disconnecting $x$ from $0$ in $B(0,4C_1)$, and noting that with positive probability the spine of $\T$ takes a given path to $x$ avoiding $\til{\T}_{-}(\ge4\til{C})$, and the trees hanging off it also avoiding $\til{\T}_{-}(\ge4\til{C})$ \footnote{We don't necessarily have $\pr{d_1^+=d_1^-=0}=\mu(1)>0$, but there exist constants $k\ge1$, $c>0$ depending only on $\sigma^2$ such that $\pr{d_1^++d_1^-\le k}\ge c$. (We can see this by noting that $\pr{d_1^++d_1^-\ge k+1}=\E{Z\1{Z\ge k+2}}\le\sqrt{\E{Z^2}\pr{Z\ge k+2}}\le\frac{\E{Z^2}}{k+2}=\frac{\sigma^2+1}{k+2}$, where $Z\sim\mu$.) Then with positive probability a given spine vertex has $\le k$ non-spine offsprings, these have no further offsprings, and the BRW at these vertices takes the same value as it does at the next spine vertex.}, we get that
\begin{align*}
\mathbb{P}\Big(&\til{\T}_{-}\left(\ge4\til{C}\right)\cap\T\left(\ge4C_1\right)=\emptyset,\:0\not\in\T_{+}\left(\ge4C_1\right),\\
&\T\left(\ge4C_1\right)\text{ does not disconnect }\Xtil_{\til{\tau}_{4\til{C}}}\text{ from }0\text{ in }B(0,4\til{C}),\\
&\til{\T}\left(\ge4\til{C}\right)\text{ does not disconnect }X_{\tau_{4C_1}}\text{ from }0\text{ in }B(0,4C_1)
\Big)\quad\lesssim_{\til{C},\:C_1}\quad\frac{1}{\log n}.
\end{align*}
This is exactly~\eqref{eq:Ttillarge_Tlarge}.

For proving~\eqref{eq:Ttillarge_Tsmall} note that for a given $n$, we have $\pr{\T\ne\T(\le n)}\lesssim\theta_1^n\ll\frac{1}{\log n}$. Using this, the bound~\eqref{eq:Ttillarge_T} with $n=\frac14C_1$, and that
the probability below is decreasing in $n$, we get that
\begin{align*}
\limsup_{n\to\infty}\:\mathbb{P}\Bigg(&\til{\T}_{-}\left(\ge4\til{C}\right)\cap\T\left(\le\frac14C_1\right)=\emptyset,\:0\not\in\T_{+}\left(\le\frac14C_1\right),\\
&\:\T\left(\le\frac14C_1\right)\text{ does not disconnect }\Xtil_{\til{\tau}_{4\til{C}}}\text{ from }0\text{ in }B\left(0,4\til{C}\right)\Bigg)\quad\lesssim_{\til{C}}\quad\frac{1}{\log C_1}.\qedhere
\end{align*}
\end{proof}

\begin{proof}[Proof of \Cref{lem:Ttilsmall_probs_bound}]
To prove~\eqref{eq:Ttilsmall_Tsmall}, firstly, we will show that
\begin{align}\label{eq:Ttilsmall_Tsmall_lim}
\pr{\til{\T}_{-}\left(\le n^{9}\right)\cap\T\left(\le n\right)=\emptyset,\:0\not\in\T_{+}\left(\le n\right)}\to0\quad\text{as }n\to0.
\end{align}

Note that the above probability can be upper bounded as
\begin{align}
\nonumber
\le\quad&\pr{\til{\T}_{-}\cap\T=\emptyset,\:0\not\in\T_{+}}\quad
+\quad\pr{\til{\T}_{-}\left(>n^{9}\right)\cap B\left(0,n^{4}\right)\ne\emptyset}\\
\label{eq:Ttilsmall_Tsmall_bound}
+\quad&\pr{\T\left(\le n\right)\not\subseteq B\left(0,n^{4}\right)}\quad
+\quad\pr{\T\ne\T\left(\le n\right)}.
\end{align}

We already know that $\pr{\til{\T}_{-}\cap\T=\emptyset,\:0\not\in\T_{-}}
\lesssim\frac{1}{\log n}$, which $\to0$ as $n\to\infty$.

Note that the second probability in~\eqref{eq:Ttilsmall_Tsmall_bound} is upper bounded by the expected number of intersections of $\til{\T}_{-}\left(>n^{9}\right)$ and $B\left(0,n^{4}\right)$, which can be upper bounded as
\begin{align*}
\le\quad\sup_{x\in\partial B\left(0,n^{9}\right)}\sum_{y\in B\left(0,n^{4}\right)} G(x,y)\quad\asymp\quad n^{4d}\cdot\frac{1}{n^{9(d-4)}}\quad\to\quad0\qquad\text{as }n\to\infty.
\end{align*}

We can upper bound the third probability in~\eqref{eq:Ttilsmall_Tsmall_bound} as
\begin{align*}
&\pr{\T\left(\le n\right)\not\subseteq B\left(0,n^{4}\right)}\quad\\
&\hspace{2cm}\le\quad\pr{\tau_{n}>n^{3}}+ n^{3}\cdot\E{d^{+}_1+d^{-}_1}\cdot\pr{\GW\text{ survives to generation }\frac12n^{4}}.
\end{align*}
From Lemma~\ref{lem:hitting_time_tail_bounds} we know that $\pr{\tau_{n}>n^{3}}\to0$ as $n\to\infty$. From \cite[Section 9, Theorem 1]{branching_processes} we know that $\pr{\GW\text{ survives to generation }k}\asymp\frac1k$, which also allows to bound the second term.

Finally, note that $\pr{\T\ne\T\left(\le n\right)}\lesssim\theta_1^n\to0$ as $n\to\infty$. This finishes the proof of~\eqref{eq:Ttilsmall_Tsmall_lim}. Using~\eqref{eq:Ttilsmall_Tsmall_lim} and that $\pr{\til{\T}_{-}\left(\le\frac14\til{C}\right)\cap\T\left(\le \frac14C_2\right)=\emptyset,\:0\not\in\T_{+}\left(\le\frac14C_2\right)}$ is decreasing in $n$, $C_2$ and $\til{C}$, we get~\eqref{eq:Ttilsmall_Tsmall}.

Now we proceed to prove~\eqref{eq:Ttilsmall_Tlarge}.
Firstly, note that
\begin{align*}
\pr{\til{\T}_{-}\left(\ge n^5\right)\cap\T\ne\emptyset}\quad\le\quad
\pr{\til{\T}_{-}\left(\ge n^5\right)\cap B\left(0,n^2\right)\ne\emptyset}+\pr{\T\not\subseteq B\left(0,n^2\right)}.
\end{align*}
The expected number of intersections of $\til{\T}_{-}\left(\ge n^5\right)$ and $B\left(0,n^2\right)$ is $\lesssim n^{2d}\cdot\frac{1}{n^{5(d-4)}}\ll\frac{1}{\log n}$. Also, $\pr{\T\not\subseteq B\left(0,n^2\right)}\le\pr{\xiln\ge n^2}+\pr{\xirn\ge n^2}\ll\frac{1}{\log n}$.

Using this and that $\pr{\til{\T}_{-}\cap\T=\emptyset,\:0\not\in\T_{+}}\lesssim\frac{1}{\log n}$ we get that
\begin{align*}
\pr{\til{\T}_{-}\left(\le n^5\right)\cap\T=\emptyset,\:0\not\in\T_{+}}\quad\lesssim\quad\frac{1}{\log n}.
\end{align*}

Then similarly to the proof of Lemma~\ref{lem:Ttillarge_probs_bound} we get that
\begin{align*}
\mathbb{P}\Big(&\til{\T}_{-}\left(\le n^5\right)\cap\T\left(\ge 4C_2\right)=\emptyset,\:0\not\in\T_{+}\left(\ge 4C_2\right),\\
&\:\til{\T}_{-}\left(\le n^5\right)\text{ does not disconnect }0\text{ and }X_{\tau_{4C_2}}\text{ in }B\left(0,4C_2\right)\Big)\quad\lesssim_{C_2}\quad\frac{1}{\log n}.
\end{align*}

Using this and that the probability in~\eqref{eq:Ttilsmall_Tlarge} is decreasing in $n$, we conclude the proof of~\eqref{eq:Ttilsmall_Tlarge}.
\end{proof}

\subsection{Concluding the proof of \Cref{pro:limsup_prob_ABD}}

By Lemmas~\ref{lem:Ttillarge_probs_bound} and~\ref{lem:Ttilsmall_probs_bound}, for any $\alpha>0$ we can choose $C_2$ and $\til{C}$, and then choose $C_1$ so that
\begin{align}\label{eq:limsup_alpha_bound}
\limsup_{n\to\infty}\:\pr{\til{\T}_{-}\left(\ge4\til{C}\right)\cap\T\left(\ge4C_1\right)=\emptyset,\:0\not\in\T_{+}\left(\ge4C_1\right),\:A(\wedge,\wedge)}\quad\le\quad\alpha,
\end{align}
and the same bound holds for the three other analogous probabilities.

Together with Lemmas~\ref{lem:bound_Dab_prob} and~\ref{lem:decorrelate_D} this shows that for the above choice of $\til{C}$, $C_1$ and $C_2$, each of~\eqref{eq:vv}, \eqref{eq:va}, \eqref{eq:av} and \eqref{eq:aa} is $\lesssim\frac{\alpha}{\log n}$, and the implicit constants in the $\lesssim$'s do not depend on $\alpha$.

This then shows that for any $\alpha>0$ we have $\pr{\A_n,\B_n,D_{\eps}}\lesssim\frac{\alpha}{\log n}$, with the implicit constant in the $\lesssim$ not depending on $\alpha$.

This concludes the proof of \Cref{pro:limsup_prob_ABD}.

\section{Proof of Theorem~\ref{thm:Ttiln_Tminusnn_prob}}\label{sec:Ttiln_Tminusnn_prob}

In this section we prove \Cref{thm:Ttiln_Tminusnn_prob}.

From \Cref{thm:prob_nocap_nozero_n} we know that $\pr{\til{\T}(-\infty,0)\cap\T[-n,n]=\emptyset,\:0\not\in\T(0,n]}\sim\frac{c_8}{\log n}$. In what follows, we show that this probability stays the same order if we remove the $0\not\in\T(0,n]$ term, and also if we then replace $\til{\T}_{-}$ by $\til{\T}[-n,0)$.

\subsection{Removing the $0\not\in\T(0,n]$ term}

This section is devoted to the proof of the first part of \Cref{thm:Ttiln_Tminusnn_prob}, i.e.\ showing that
\begin{align}\label{eq:Ttilminus_Tminusnn_prob}
\pr{\til{\T}(-\infty,0)\cap\T[-n,n]=\emptyset}\quad\asymp\quad\frac{1}{\log n}.
\end{align}

The lower bound is immediate from \Cref{thm:prob_nocap_nozero_n}. The proof of the upper bound will rely on the following two results.

Firsly, instead of $\T$ we consider a tree that does not hit 0 at positive times, but is weighted by the number of zeroes at non-positive times. We show the following.

\begin{lemma}\label{lem:Ttilminus_Tminusnn_prob_as_E}
\begin{align*}
&\pr{\til{\T}(-\infty,0)\cap\T[-n,n]=\emptyset}\\
&\hspace{2cm}\sim\:\E{\1{\til{\T}(-\infty,0)\cap\T[-n,n]=\emptyset}\1{0\not\in\T(0,\infty)}\#\left(0\text{ in }\T(-\infty,0]\right)}+o\left(\frac{1}{\log n}\right).
\end{align*}
\end{lemma}

Secondly, we upper bound an expectation similar to the above one, but instead of considering the tree $\T$ up to a given number of vertices in the past and the future, we consider it up to a given number of vertices along the spine. For an interval $I$, we write $\T^I$ to mean the spine vertices of $\T$ with indices in $I$, together with the trees hanging off them. We define $\T_{+}^I$ analogously, but only keeping the parts with positive indices, and $\T_{-}^I$ analogously, only keeping the negative indices. We show the following.

\begin{lemma}\label{lem:E_noncap_count_zeroes}
\begin{align*}
\E{\1{\til{\T}_{-}\cap\T^{[0,2m]}=\emptyset}\1{0\not\in\T_{+}^{[0,2m]}}\:\#\left(0\text{ in }\T_{-}^{[0,m]}\right)}\quad\lesssim\quad\frac{1}{\log m}.
\end{align*}
\end{lemma}

We start by presenting the proofs of Lemmas~\ref{lem:Ttilminus_Tminusnn_prob_as_E} and~\ref{lem:E_noncap_count_zeroes}, and then explain how to conclude the proof of \eqref{eq:Ttilminus_Tminusnn_prob} from these.

\begin{proof}[Proof of \Cref{lem:Ttilminus_Tminusnn_prob_as_E}]

By considering whether $\T\left[n^{\frac14},n\right]$ hits 0, and if not, considering the last visit of $\T\left[0,n^{\frac14}\right]$ at 0, we can write down an upper bound as follows,
\begin{align}\label{eq:last_0_decomp}
&\pr{\til{\T}_{-}\cap\T[-n,n]=\emptyset}\quad\\
\nonumber
&\le\quad\pr{0\in\T\left[n^{\frac14},n\right]}+
\sum_{j=0}^{n^{\frac14}}\pr{\til{\T}_{-}\cap\T[-n,n]=\emptyset,\:\T_j=0,\:0\not\in\T\left(j,n\right]}.
\end{align}

The first term on the right-hand side of~\eqref{eq:last_0_decomp} can be bounded as
\begin{align}
\nonumber
\pr{0\in\T\left[n^{\frac14},n\right]}\quad\le\quad
&\pr{0\in\T_+\left(\ge\log n\right)}+\pr{\tau_{\log n}\ge (\log n)^4}\\
\label{eq:prob_far_0_bound}
+&\pr{\sum_{i=1}^{(\log n)^4}d_i^+\ge(\log n)^6}+\pr{\sum_{i=1}^{(\log n)^6}|{\GW}_i|>n^{\frac14}}.
\end{align}
Using that $G(x)\asymp\frac{1}{1+||x||^4}$, that $\E{d_i^+}\asymp\Var{d_i^+}\asymp1$, and Lemmas~\ref{lem:sum_GW_sizes} and~\ref{lem:hitting_time_tail_bounds}, we see that each of the above probabilities is $\ll\frac{1}{\log n}$.

Using \Cref{lem:shift_invariance}, the sum in~\eqref{eq:last_0_decomp} can be bounded as follows.
\begin{align}
\nonumber
\sum_{j=0}^{n^{\frac14}}\:&\pr{\til{\T}_{-}\cap\T[-n,n]=\emptyset,\:\T_j=0,\:0\not\in\T\left(j,n\right]}\\
\nonumber
&=\quad\sum_{j=0}^{n^{\frac14}}\pr{\til{\T}_{-}\cap\T[-n-j,n-j]=\emptyset,\:\T_{-j}=0,\:0\not\in\T\left(0,n-j\right]}\\
\nonumber
&\le\quad\sum_{j=0}^{n^{\frac14}}\pr{\til{\T}_{-}\cap\T\left[-n,n-n^{\frac14}\right]=\emptyset,\:\T_{-j}=0,\:0\not\in\T\left(0,n-n^{\frac14}\right]}\\
\label{eq:last_0_upper_bound}
&=\quad\E{\1{\til{\T}_{-}\cap\T\left[-n,n-n^{\frac14}\right]=\emptyset}\1{0\not\in\T\left(0,n-n^{\frac14}\right]}\:\#\left(0\text{ in }\T\left[-n^{\frac14},0\right]\right)}.
\end{align}

Analogously, we get
\begin{align}
\nonumber
\pr{\til{\T}_{-}\cap\T[-n,n]=\emptyset}\quad\ge&\quad\sum_{j=0}^{n^{\frac14}}\pr{\til{\T}_{-}\cap\T[-n,n]=\emptyset,\:\T_j=0,\:0\not\in\T\left(j,n\right]}\\
\label{eq:last_0_lower_bound}
\ge&\quad\E{\1{\til{\T}_{-}\cap\T\left[-n-n^{\frac14},n\right]=\emptyset}\1{0\not\in\T\left(0,\infty\right)}\:\#\left(0\text{ in }\T\left[-n^{\frac14},0\right]\right)}.
\end{align}

Then by simple comparisons we can show that the expectations~\eqref{eq:last_0_upper_bound} and~\eqref{eq:last_0_lower_bound} differ from the expectation in \Cref{lem:Ttilminus_Tminusnn_prob_as_E} by $o\left(\frac{1}{\log n}\right)$. This finishes the proof of \Cref{lem:bound_Dab_prob}.
\end{proof}

\begin{proof}[Proof of \Cref{lem:E_noncap_count_zeroes}]

We will split the counting term $\#\left(0\text{ in }\T_{-}^{[0,m]}\right)$ based on the trees descending from different spine vertices, and use different estimates based on the value of the walk at the given spine vertex and the number of times that value has been visited by earlier spine vertices.

For a given $k\ge1$ let $\tau^j_k$ be the $j$th index $i$ such that $X_i\in\partial B(0,k)$, and for a given $x\in\Z^d$ let $\tau^j_x$ be the $j$th index $i$ such that $X_i=x$. We set $\tau^j_k$ or $\tau^j_x$ to be $\infty$ if such $i$ does not exist. Then for any positive constants $c$ and $C$ we can bound the expectation in \Cref{lem:E_noncap_count_zeroes} as follows.
\begin{align*}
\E{\1{\til{\T}_{-}\cap\T^{[0,2m]}=\emptyset}\1{0\not\in\T_{+}^{[0,2m]}}\:\#\left(0\text{ in }\T_{-}^{[0,m]}\right)}\:=\:\sum_{\ell=0}^{m}\E{\1{\til{\T}_{-}\cap\T^{[0,2m]}=\emptyset}\1{0\not\in\T_{+}^{[0,2m]}}\:\#\left(0\text{ in }\T_{-}^{\ell}\right)}
\end{align*}
\begin{align}
\label{eq:E_numzeroes_x_many_returns}
&\le\quad\sum_{x}&\sum_{j\ge C\log(||x||)} &\E{\1{\tau^j_x\le m}\:\1{\til{\T}_{-}\cap\T^{[0,2m]}=\emptyset}\1{0\not\in\T_{+}^{[0,2m]}}\:\#\left(0\text{ in }\T_{-}^{\tau^j_x}\right)}\\
\label{eq:E_numzeroes_k_few_returns}
&+\quad\sum_{k\ge1}&\sum_{j\le k^4}\quad &\E{\1{\tau^j_k\le m}\:\1{\til{\T}_{-}\cap\T^{[0,2m]}=\emptyset}\1{0\not\in\T_{+}^{[0,2m]}}\:\#\left(0\text{ in }\T_{-}^{\tau^j_k}\right)}\\
\label{eq:E_numzeroes_x_few_returns}
&+\quad\sum_{x}&\sum_{j\le C\log(||x||)} &\E{\1{\tau^j_x\le m}\:\1{\til{\T}_{-}\cap\T^{[0,2m]}=\emptyset}\1{0\not\in\T_{+}^{[0,2m]}}\:\#\left(0\text{ in }\T_{-}^{\tau^j_x}\right)\:\1{\#\left(0\text{ in }\T_{-}^{\tau^j_x}\right)\ge e^{c||x||^3}}}\\
\label{eq:E_numzeroes_k_many_returns}
&+\quad\sum_{k\ge1}&\sum_{j\ge k^4}\quad &\E{\1{\tau^j_k\le m}\:\1{\til{\T}_{-}\cap\T^{[0,2m]}=\emptyset}\1{0\not\in\T_{+}^{[0,2m]}}\:\#\left(0\text{ in }\T_{-}^{\tau^j_k}\right)\:\1{\#\left(0\text{ in }\T_{-}^{\tau^j_k}\right)\le e^{ck^3}}}.
\end{align}

Now we discuss how to bound each of the above sums.

To bound~\eqref{eq:E_numzeroes_k_few_returns}, note that
\begin{align*}
&\E{\1{\tau^j_k\le m} \1{\til{\T}_{-}\cap\T^{[0,2m]}=\emptyset}\1{0\not\in\T_{+}^{[0,2m]}}\:\#\left(0\text{ in }\T_{-}^{\tau^j_k}\right)}\quad\\
&\le\quad
\E{\1{\tau^j_k\le m}\1{\til{\T}_{-}\cap\T^{[0,2m]\setminus\{\tau^j_k\}}=\emptyset}\1{0\not\in\T_{+}^{[0,2m]\setminus\{\tau^j_k\}}}\:\#\left(0\text{ in }\T_{-}^{\tau^j_k}\right)}\\
&\asymp\quad\E{\1{\tau^j_k\le m}\1{\til{\T}_{-}\cap\T^{[0,2m]\setminus\{\tau^j_k\}}=\emptyset}\1{0\not\in\T_{+}^{[0,2m]\setminus\{\tau^j_k\}}}\:g\left(X_{\tau^j_k}\right)}\quad\\
&\lesssim\quad
\frac{1}{k^6}\:\E{\1{\tau^j_k\le m}\1{\til{\T}_{-}\cap\T^{[0,2m]}=\emptyset}\1{0\not\in\T_{+}^{[0,2m]}}}.
\end{align*}
In the last $\lesssim$ we used that with positive probability the spine vertex of index $\tau^j_k$ has a bounded number of offsprings which have no further offsprings, and the BRW at these vertices takes the same value as it does at the next spine vertex.

This gives
\begin{align*}
\eqref{eq:E_numzeroes_k_few_returns}\quad\lesssim\quad\sum_{k\ge1}\frac{k^4}{k^6}\:\E{\1{\tau^j_k\le m}\1{\til{\T}_{-}\cap\T^{[0,2m]}=\emptyset}\1{0\not\in\T_{+}^{[0,2m]}}}\quad\lesssim\quad\sum_{k\ge1}\frac{1}{k^2}\cdot\frac{1}{\log m}\quad\asymp\quad\frac{1}{\log m}.
\end{align*}

To bound~\eqref{eq:E_numzeroes_x_many_returns}, we pull out the term $\E{\#\left(0\text{ in }\T^{\tau^j_x}_{-}\right)}$ as above, and consider the tree obtained by disregarding the part of $\T$ between the spine vertices of indices $\tau_x$ and $\tau^j_x$. This gives
\begin{align*}
\eqref{eq:E_numzeroes_x_many_returns}\quad\lesssim\quad\sum_{x}\sum_{j\ge C\log(||x||)}\prstart{\tau_0<\infty}{0}^{j-1}\pr{\tau_x\le m,\:\til{\T}_{-}\cap\T^{[0,m]}=\emptyset,\:0\not\in\T_{+}^{[0,m]}}g(x)\\
\lesssim\quad\sum_{k\ge1}k^7\cdot \prstart{\tau_0<\infty}{0}^{C\log k}\cdot\frac{1}{\log m}\cdot\frac{1}{k^6}.
\end{align*}
If $C$ is sufficiently large, then $\prstart{\tau_0<\infty}{0}^{C\log k}\lesssim\frac{1}{k^3}$. Then the above sum is $\lesssim\frac{1}{\log m}$.

To bound~\eqref{eq:E_numzeroes_x_few_returns}, we again consider a tree obtained by omitting the part of $\T$ between spine vertices $\tau_x$ and $\tau^j_x$, and we pull out the term $\E{\#\left(0\text{ in }\T^{\tau^j_x}_{-}\right)\1{\#\left(0\text{ in }\T^{\tau^j_x}_{-}\right)\ge e^{c||x||^3}}}$. This expectation can be bounded as
\begin{align*}
\E{\#\left(0\text{ in }\T^{\tau^j_x}_{-}\right)\1{\#\left(0\text{ in }\T^{\tau^j_x}_{-}\right)\ge e^{c||x||^3}}}
\quad&=\quad \sum_{s\ge e^{c||x||^3}}\pr{\#\left(0\text{ in }\T^{\tau^j_x}_{-}\right)\ge s}\\
&\le\quad \sum_{s\ge e^{c||x||^3}}
\frac{\E{\#\left(0\text{ in }\T^{\tau^j_x}_{-}\right)^2}}{s^2}\quad\lesssim\quad e^{-c||x||^3}.
\end{align*}
This then gives
\begin{align*}
\eqref{eq:E_numzeroes_x_few_returns}\:\lesssim\:\sum_{k\ge1} k^7\cdot C(\log k)\cdot \sup_{x\in\partial B(0,k)}\pr{\tau_x\le m,\:\til{\T}_{-}\cap\T^{[0,m]}=\emptyset,\:0\not\in\T_{+}^{[0,m]}}\cdot e^{-c||x||^3}\:
\lesssim\:\frac{1}{\log m}.\quad&
\end{align*}

Finally, we bound~\eqref{eq:E_numzeroes_k_many_returns} as follows.

Let $\what{\T}$ be a BRW on an infinite random tree that is distributed like $\T$, but with the root having the same offspring distribution on the two sides of the spine as the later spine vertices. Let $T_i$ be BRWs on $\GW(\mu)$ trees. Let $e$ and $e_i$ be uniform neighbours of 0 in $\Z^8$. Let all of these variables be independent of each other and $\til{\T}$.

By considering the part of $\T$ after the spine vertex $\tau^j_k$, we get that
\begin{align}
\nonumber
&\sum_{j\ge k^4}\E{\1{\til{\T}_{-}\cap\T^{[0,2m]}=\emptyset}\1{0\not\in\T_{+}^{[0,2m]}}\1{\tau^j_k\le m}\:\#\left(0\text{ in }\T_{-}^{\tau^j_k}\right)\1{\#\left(0\text{ in }\T_{-}^{\tau^j_k}\right)\le e^{ck^3}}}\\
\nonumber
&\lesssim\:\sum_{j\ge k^4}\pr{\tau^j_k\le m}\cdot\\
\nonumber
&\quad\cdot\sup_{x\in\partial B(0,k)}\mathbb{E}\Bigg[\1{x\not\in\til{\T}_{-}}\cdot\prcond{\til{\T}_{-}\cap \left(x+e+\what{\T}^{[0,m]}\right)=\emptyset,\:0\not\in \left(x+e+\what{\T}_{+}^{[0,m]}\right)}{\til{\T}_{-}}\cdot\\
\label{eq:T_from_x}
&\hspace{2cm}\cdot\econd{\sum_{i=1}^{d_1^+}\#\left(0\text{ in }(x+e_i+T_i)\right)\1{\#\left(0\text{ in }(x+e_i+T_i)\right)\le e^{ck^3}}\1{\til{\T}_{-}\cap(x+e_i+T_i)=\emptyset}}{\til{\T}_{-}}\Bigg].
\end{align}
We can see that $\pr{\tau^j_k<m}\le\left(\sup\limits_{x\in\partial B(0,k)}\prstart{\tau_k<\infty}{x}\right)^j$. Using this and \Cref{lem:prob_returning_to_partialBk} we get that
\begin{align*}
\sum_{j\ge k^4}\pr{\tau^j_k<m}\quad\lesssim\quad \left(1-\frac{c'}{k}\right)^{k^4}k\quad\lesssim\quad e^{-c'' k^3},
\end{align*}
where $c'$ and $c''$ are small positive constants.

Now we explain how to bound the expectation on the right-hand side of~\eqref{eq:T_from_x} by comparing it to $\pr{\til{\T}_{-}\cap\T^{[0,m]},\:0\not\in\T^{[0,m]}_{+}}$.

We will use the following observation. Let $t$ be any finite plane tree and let $v$ be any leaf of $t$ at distance $\ell$ from the root. Let $\what{\T}$ be as above and let $T$ denote a (planar) $\GW(\mu)$ tree and a random walk indexed by this tree. Then the following two events have the same probability, $\prod_{u\in t\setminus\{v\}}\mu\left(\deg_t(u)\right)$. (i) The tree $T$ is isomorphic to a tree that is obtained from $t$ by attaching to it a (possibly empty) tree from $v$. (ii) The subtree of $\what{\T}$ formed by the spine vertices of indices $[0,\ell]$ and the trees on the two sides of the spine descending from the spine vertices of indices $[0,\ell-1]$ is isomporphic to $t$, with the $\ell$th spine vertex corresponding to $v$.

By considering the first 0 in $(x+T)$ and using the above property, we can see that
\begin{align*}
&\prcond{0\in(x+T),\:\#\left(0\text{ in }(x+T)\right)\le e^{ck^3},\:\til{\T}_{-}\cap(x+T)=\emptyset}{\til{\T}_{-}}\\
&\:\le\quad\mathbb{P}\Big((x+\what{X})\text{ hits }0,\:0\not\in(x+\what{\T}_{+})^{[0,{\tau}_0-1]},\:\#\left(0\text{ in }(x+\what{\T}_{-})^{[0,{\tau}_0-1]}\right)\le e^{ck^3}-1,\\
&\hspace{8cm}\:\til{\T}_{-}\cap\left((x+\what{\T})^{[0,{\tau}_0-1]}\cup\{0\}\right)=\emptyset\:\Big|\:\til{\T}_{-}\Big)\\
&\:\lesssim\quad\mathbb{P}\Big((x+\what{X})\text{ hits }0,\:0\not\in(x+\what{\T}_{+})^{[0,{\tau}_0]},\:\#\left(0\text{ in }(x+\what{\T}_{-})^{[0,{\tau}_0]}\right)\le e^{ck^3},\\
&\hspace{9cm}\:\til{\T}_{-}\cap(x+\what{\T})^{[0,{\tau}_0]}=\emptyset\:\Big|\:\til{\T}_{-}\Big),
\end{align*}
where $\tau_0$ refers to the first time $(x+\what{X})$ hits 0.
If $\eta=\max\{i\in[0,{\tau}_0]:\:(x+\what{X})(i)=x\}$ then by restricting to $(x+\what{\T})^{[\eta,{\tau}_0]}$ and viewing this tree and BRW from the ${\tau}_0$th spine vertex, we can upper bound the above probability as
\begin{align*}
&\le\quad\mathbb{P}\Big((x+\what{X})\text{ hits }0,\:0\not\in(x+\what{\T}_{+})^{[\eta,{\tau}_0]},\:\#\left(0\text{ in }(x+\what{\T}_{-})^{[\eta,{\tau}_0]}\right)\le e^{ck^3},\\
&\hspace{9.5cm}\:\til{\T}_{-}\cap(x+\what{\T})^{[\eta,{\tau}_0]}=\emptyset\:\Big|\:\til{\T}_{-}\Big)\\
&=\quad\prcond{\what{\tau}_x<\infty,\:0\not\in\what{X}(0,\what{\tau}_x],\:0\not\in\what{\T}_{+}^{[0,\what{\tau}_x]},\:\#\left(0\text{ in }\what{\T}_{-}^{[0,\what{\tau}_x]}\right)\le e^{ck^3},\:\til{\T}_{-}\cap\what{\T}^{[0,\what{\tau}_x]}=\emptyset}{\til{\T}_{-}}\\
&\le\quad\prcond{\what{\tau}_x<\infty,\:0\not\in\what{\T}_{+}^{[0,\what{\tau}_x]},\:\#\left(0\text{ in }\what{\T}_{-}^{[0,\what{\tau}_x]}\right)\le e^{ck^3},\:\til{\T}_{-}\cap\what{\T}^{[0,\what{\tau}_x]}=\emptyset}{\til{\T}_{-}}.
\end{align*}
Also, by the Paley-Zygmund inequality we have that
\begin{align*}
&\prcond{0\in(x+T),\:\#\left(0\text{ in }(x+T)\right)\le e^{c||x||^3},\:\til{\T}_{-}\cap(x+T)=\emptyset}{\til{\T}_{-}}\\
&\quad\ge\quad\frac{\econd{\#\left(0\text{ in }(x+T)\right)\1{\#\left(0\text{ in }(x+T)\right)\le e^{ck^3}}\1{\til{\T}_{-}\cap(x+T)=\emptyset}}{\til{\T}_{-}}^2}{\econd{\#\left(0\text{ in }(x+T)\right)^2\1{\#\left(0\text{ in }(x+T)\right)\le e^{ck^3}}\1{\til{\T}_{-}\cap(x+T)=\emptyset}}{\til{\T}_{-}}}\\
&\quad\ge\quad e^{-ck^3}\:\econd{\#\left(0\text{ in }(x+T)\right)\1{\#\left(0\text{ in }(x+T)\right)\le e^{ck^3}}\1{\til{\T}_{-}\cap(x+T)=\emptyset}}{\til{\T}_{-}}.
\end{align*}
Combining these and using that $\E{d_1^+}\asymp1$, and with positive probability the trees hanging off a given spine vertex do not produce values for the BRW that are not already present at the rest of the tree, we get that
\begin{align*}
&\mathbb{E}\Bigg[\1{x\not\in\til{\T}_{-}}\cdot\prcond{\til{\T}_{-}\cap \left(x+e+\what{\T}^{[0,m]}\right)=\emptyset,\:0\not\in \left(x+e+\what{\T}_{+}^{[0,m]}\right)}{\til{\T}_{-}}\cdot\\
&\hspace{2cm}\cdot\econd{\sum_{i=1}^{d_1^+}\#\left(0\text{ in }(x+e_i+T_i)\right)\1{\#\left(0\text{ in }(x+e_i+T_i)\right)\le e^{ck^3}}\1{\til{\T}_{-}\cap(x+e_i+T_i)=\emptyset}}{\til{\T}_{-}}\Bigg]\\
&\lesssim\quad e^{ck^3}\:\pr{\what{\tau}_x<\infty,\:0\not\in\what{\T}_{+}^{[0,\what{\tau}_x+m]},\:\til{\T}_{-}\cap\what{\T}^{[0,\what{\tau}_x+m]}=\emptyset}\\
&\hspace{4.5cm}\lesssim\quad e^{ck^3}\: \pr{0\not\in{\T}_{+}^{[0,m]},\:\til{\T}_{-}\cap{\T}^{[0,m]}=\emptyset} \quad\lesssim\quad e^{ck^3}\:\frac{1}{\log m}\:.
\end{align*}

Overall this shows that for sufficiently small values of $c$ we have
\begin{align*}
\eqref{eq:E_numzeroes_k_many_returns}\quad\lesssim\quad\sum_{k\ge1} e^{ck^3}\: e^{-c''k^3}\:\frac{1}{\log m}\quad\lesssim\quad\frac{1}{\log m}.
\end{align*}

This finishes the proof of \Cref{lem:E_noncap_count_zeroes}.
\end{proof}

\begin{proof}[Proof of \eqref{eq:Ttilminus_Tminusnn_prob}]
As mentioned earlier, the lower bound is immediate from \Cref{thm:prob_nocap_nozero_n}.

For the upper bound it is sufficient to show that the expectation in \Cref{lem:Ttilminus_Tminusnn_prob_as_E} is $\lesssim\frac{1}{\log n}$.

We can write
\begin{align}
\nonumber
&\E{\1{\til{\T}_{-}\cap\T[-n,n]=\emptyset}\1{0\not\in\T(0,\infty)}\#\left(0\text{ in }\T(-\infty,0]\right)}\\
\label{eq:E_noncap_nu0_bound}
&\le\:\E{\1{\til{\T}_{-}\cap\T^{[0,n^{1/4}]}=\emptyset}\1{0\not\in\T_{+}}\:\#\left(0\text{ in }\T(-\infty,0]\right)}+
\E{\1{\T\left[-n,n\right]\not\supseteq\T^{[0,n^{1/4}]}}\:\#\left(0\text{ in }\T(-\infty,0]\right)}.
\end{align}

Using \Cref{lem:E_noncap_count_zeroes} with $m=n^{\frac14}$, and using that $\E{\#\left(0\text{ in }\T(-\infty,m]\right)}\ll\frac{1}{\log m}$, and $\pr{0\in\T^{[m,\infty)}_{+}}\ll\frac{1}{(\log m)^2}$, we get that the first term in~\eqref{eq:E_noncap_nu0_bound} is $\lesssim\frac{1}{\log n}$.

The second term in~\eqref{eq:E_noncap_nu0_bound} can be bounded by Cauchy-Schwarz as
\begin{align*}
\le \sqrt{\pr{\T\left[-n,n\right]\not\supseteq\T^{[0,n^{1/4}]}}\E{\#\left(0\text{ in }\T\left(-\infty,0\right]\right)^2}}.
\end{align*}
Here the probability is $\ll\frac{1}{(\log n)^3}$, while the expectation is $\asymp1$.

This finishes the proof of \eqref{eq:Ttilminus_Tminusnn_prob}.
\end{proof}

\subsection{Replacing $\til{\T}_{-}$ with $\til{\T}[-n,0)$}

In this section we prove the second part of \Cref{thm:Ttiln_Tminusnn_prob}.

Note that
\begin{align*}
&\pr{\til{\T}(-\infty,0)\cap\T[-n,n]=\emptyset}\:\le\:\pr{\til{\T}[-n,0)\cap\T[-n,n]=\emptyset}\\
&\hspace{2cm}\le\:\pr{\til{\T}[-n,0)\cap\T[-n^{\frac{1}{17}},n^{\frac{1}{17}}]=\emptyset}\\
&\hspace{2cm}\le\:\pr{\til{\T}(-\infty,0)\cap\T[-n^{\frac{1}{17}},n^{\frac{1}{17}}]=\emptyset}+\pr{\til{\T}(-\infty,-n)\cap\T[-n^{\frac{1}{17}},n^{\frac{1}{17}}]\ne\emptyset}.
\end{align*}
By \eqref{eq:Ttilminus_Tminusnn_prob} we have $\pr{\til{\T}(-\infty,0)\cap\T[-n,n]=\emptyset}\asymp\pr{\til{\T}(-\infty,0)\cap\T[-n^{\frac{1}{17}},n^{\frac{1}{17}}]\ne\emptyset}\asymp\frac{1}{\log n}$.
We can bound $\pr{\til{\T}(-\infty,-n)\cap\T[-n^{\frac{1}{17}},n^{\frac{1}{17}}]=\emptyset}$ as follows. Let $m=n^{\frac{1}{17}}$.
Note that $\T[-m,m]\subseteq B(0,m)$, hence
\begin{align*}
&\pr{\til{\T}(-\infty,-m^{17})\cap\T[-m,m]\ne\emptyset}\quad\le\quad\pr{\til{\T}(-\infty,-m^{17})\cap B(0,m)\ne\emptyset}\\
&\le\quad\pr{\til{\T}_{-}\left(\ge m^{3}\right)\cap B(0,m)\ne\emptyset}\:+\: \pr{\tau_{m^{3}}\ge m^{7}}\\
&\hspace{5cm}+\pr{\sum_{i=1}^{m^7}d_i^-\ge m^8}\:+\:\pr{\sum_{i=1}^{m^{8}}(|{\GW}_i|+1)\ge m^{17}}.
\end{align*}
Each of the terms on the right-hand side is $\ll\frac{1}{\log m}$. This finishes the proof.

\section{Proof of Theorem~\ref{thm:BCap_asymp}}\label{sec:BCap_asymp}

Using the definition of $\BCap{\cdot}$, \Cref{lem:shift_invariance}, the property that $\left(\xiln\mid\xiln+\xirn=m\right)\sim\Unif{\left\{0,1,...,m\right\}}$, and \Cref{thm:prob_nocap_nozero}, we can see that
\begin{align*}
&\E{\frac{\BCap{\T\left[0,\xiln+\xirn\right]}}{\xiln+\xirn+1}}\quad=\quad\sum_{m\ge0}\pr{\xiln+\xirn=m}\frac{1}{m+1}\E{\BCap{\T[0,m]}}\\
&=\quad\sum_{m\ge0}\pr{\xiln+\xirn=m}\frac{1}{m+1}\sum_{i=0}^{m}\pr{\T_i\not\in\T(i+1,m],\:\left(\T_i+\til{\T}_{-}\right)\cap\T[0,m]=\emptyset}\\
&=\quad\sum_{m\ge0}\pr{\xiln+\xirn=m}\frac{1}{m+1}\sum_{i=0}^{m}\pr{\T_0\not\in\T(0,m-i],\:\til{\T}_{-}\cap\T[-i,m-i]=\emptyset}\\
&=\quad\pr{0\not\in\T(0,\xirn],\:\til{\T}(-\infty,0)\cap\T[-\xiln,\xirn]=\emptyset}\quad\sim\quad\frac{c_8}{\log n}.
\end{align*}

Using this, we can bound $\E{\BCap{\T[0,n]}}$ as follows.

For notational convenience write $\zeta_k=\xi^\ell_k+\xi^r_k$ and $\B(k)=\E{\BCap{\T[0,k]}}$. By the subadditivity of $\BCap{\cdot}$ and the shift invariance of $\T$ we have $\frac{\B(k\ell)}{k\ell}\le\frac{\B(k)}{k}$ for any $k,\ell\in\Zpos$, hence $\frac{\B(n)}{n}\le\frac{\B(k)}{k}\frac{n+k}{n}$ for any $n,k\in\Zpos$. 

Using this, the simple bound $\B(k)\le k+1$, the concentration of $\zeta_k$, and the already established bound on $\E{\frac{\B(\zeta_k)}{k+1}}$, we can write
\begin{align*}
&\frac{\B(n)}{n}\quad\le\quad\pr{\zeta_{\frac{n}{\log n}}>\frac{n}{\sqrt{\log n}}}\frac{n+1}{n}+ \E{\frac{\B\left(\zeta_{\frac{n}{\log n}}\right)}{\zeta_{\frac{n}{\log n}}}}\frac{n+\frac{n}{\sqrt{\log n}}}{n}\quad\sim\quad\frac{c_8}{\log n}\:,\\
&\frac{c_8}{\log n}\quad\sim\quad\E{\frac{\B\left(\zeta_{n(\log n)^3}\right)}{\zeta_{n(\log n)^3}+1}}\quad\le\quad\pr{\zeta_{n(\log n)^3}\le n\log n}+\frac{\B(n)}{n}\frac{n\log n+n}{n\log n}\\
&\hspace{10.5cm}\sim\quad\frac{\B(n)}{n}+o\left(\frac{1}{\log n}\right).
\end{align*}

This finishes the proof of Theorem~\ref{thm:BCap_asymp}.

\section{A note on the choice of $\mu$}\label{sec:mu}

We would like to emphasise that in Theorems~\ref{thm:Ttiln_Tminusnn_prob}, \ref{thm:prob_nocap_nozero_n}, \ref{thm:BCap_asymp} and~\ref{thm:prob_nocap_nozero} the offspring distribution $\mu$ is the same for all $n$. While the constants $c$, $C$ and $c_8$ only depend on $\sigma^2$, the choice of sufficiently large $n$ in~\Cref{thm:Ttiln_Tminusnn_prob}, the $o(1)$ in Theorems~\ref{thm:prob_nocap_nozero_n} and~\ref{thm:prob_nocap_nozero}, and the convergence in~\Cref{thm:BCap_asymp} can depend on $\mu$.

As an illustration we give an example of a sequence $(\mu_n)$ such that each $\mu_n$ has mean 1 and variance $\sigma^2$, but the statements of Lemmas~\ref{lem:GW_tail_bound} and~\ref{lem:theta_fraction} are not satisfied.

\begin{example}
Let $\sigma^2=2$, $\mu(n+1)=\frac{2}{n(n+1)}$, $\mu(1)=1-\frac{2}{n}$, $\mu(0)=\frac{2}{n+1}$. Then we have
\begin{align*}
\pr{|\GW|\ge n}\quad\ge\quad\pr{W_1=...=W_n=0}\quad=\quad\left(1-\frac{2}{n}\right)^{n}\quad\asymp\quad1,
\end{align*}
so \Cref{lem:GW_tail_bound} does not hold.

Also, in this case $\pr{d_1^+=\ell}=\frac{2}{n(n+1)}$ for $\ell\in\{1,2,...,n\}$ and $\pr{d_1^+=0}=1-\frac{2}{(n+1)}$, hence
\begin{align*}
\E{\sum_{\ell=0}^{d_1^+-1}\til{\theta}^{\ell}}\quad\le\quad\E{d_1^+}-(1-\til{\theta}^{\frac{n}{2}})\frac{n}{2}\pr{d_1^+>\frac{n}{2}}\quad=\quad\E{d_1^+}-\Theta(1),
\end{align*}
so \Cref{lem:theta_fraction} does not hold either.
\end{example}

It might be possible to impose some additional condition on the sequence $(\mu_n)$ (e.g.\ existence of some higher moments) that would be sufficient for Theorems~\ref{thm:Ttiln_Tminusnn_prob}, \ref{thm:prob_nocap_nozero_n}, \ref{thm:BCap_asymp} and~\ref{thm:prob_nocap_nozero} to hold, but we did not investigate this further.

As far as we understand, previous results also concern $\mu$ that is the same for all $n$.

We also note, that while our results are presented with a single offspring distribution $\mu$, they still hold and the proofs work the same way if the offpsring distribution $\til{\mu}$ associated with $\til{\T}$ and $\BCap{\cdot}$ is different from the offspring distribution $\mu$ associated with $\T$. In this case the constants depend on the variance of both $\mu$ and $\til{\mu}$.

\section*{Acknowledgements}

I would like to thank Perla Sousi for the very helpful discussions and feedback.

I was supported by the DPMMS EPSRC DTP.

\appendix

\section{Estimates regarding $|\GW|$, $\LW$ and $\LH$}\label{app:GW_LW_LH}

\begin{lemma}\label{lem:GW_tail_bound} A critical Galton-Watson tree $\GW$ satisfies
\[\pr{|\GW|\ge k}\:\sim\:\frac{c}{\sqrt{k}}\qquad\text{and}\qquad\E{|\GW|\1{|\GW|\le k}}\:\sim\:c\sqrt{k},\]
where $c$ is a positive constant depending only on $\sigma^2$.
\end{lemma}

\begin{proof}
For the first result, first note that \[\pr{|\GW|=k}\quad=\quad\pr{W_k=-1,\:W_1,...,W_{k-1}\ge0}\quad=\quad\frac{1}{k}\pr{W_k=-1}.\]
(The second equality follows e.g.\ from applying \cite[Theorem 1]{ballot_problems} to the increments of $-W$.)
By~\cite[Theorem 2.3.9]{RW_modern_intro} we know that $\pr{W_k=-1}\sim\frac{c'}{\sqrt{k}}$, where $c'$ is a positive constant depending only on $\sigma^2$. This then gives
$\pr{|\GW|\ge k}\sim\sum_{\ell\ge k}\frac{c'}{\ell^{\frac32}}\sim\frac{c}{\sqrt{k}}$ as required.

Writing $\E{|\GW|\1{|\GW|\le k}}=\sum_{j=1}^{k}\pr{|\GW|\ge j}-k\pr{|\GW|\ge k+1}$ and using the first result, we also get the second result.
\end{proof}

\begin{lemma}\label{lem:LW_LH_tail_bound}
The ladder width $\LW$ and ladder height $\LH$ corresponding to a critical walk with steps $\sim(\mu-1)$ satisfy
\begin{align*}\pr{\LW\ge k}\:&\asymp\:\frac{1}{\sqrt{k}}\qquad\text{and}\qquad\E{\LW\1{\LW\le k}}\:\asymp\:\sqrt{k},\\
\pr{\LH\ge k}\:&\lesssim\:\frac{1}{k}\qquad\text{and}\qquad\E{\LH\1{\LH\le k}}\:\asymp\:1,
\end{align*}
where the implicit constants in the $\asymp$'s and the $\lesssim$ only depend on $\sigma^2$.
\end{lemma}

\begin{proof}
We have
\[
\pr{\LW\ge k}\:=\:\pr{W_t<0\text{ for }t\in\{1,2,...,k-1\}}\:\le\:\pr{W_t\ne0\text{ for }t\in\{1,2,...,k-1\}}.
\]
From~\cite[Theorem 2.26]{MCMT} we know that this is $\le\frac{4\sigma}{\sqrt{k-1}}$.

For the other bounds note that
\begin{align*}
\pr{\LW=k,\:\LH=\ell}\quad&=\quad\pr{\text{the records on }[0,k]\text{ are }0\text{ and }k;\:W_k=\ell}\\
&=\quad\pr{\text{the right minima on }[0,k]\text{ are }0\text{ and }k;\:W_k=\ell}\\
&=\quad\pr{v_k\text{ is a child of }v_0;\text{ there are }\ell\text{ later children}}.
\end{align*}
This means that $\pr{\LH=\ell}=\mu(\ge\ell+1)$, so using that $\mu$ has finite second moment we get $\E{\LH}\asymp 1$ and $\E{\LH\1{\LH\le k}}\asymp 1$ for all $k$. This also immediately implies $\pr{\LH\ge k}\lesssim\frac{1}{k}$.

Finally, note that
\begin{align}
\nonumber
\E{\LW\1{\LW\le n}}\quad&=\quad\sum_{k=1}^{n}k\:\pr{v_k\text{ is a child of }v_0}\\
\nonumber
&=\quad\sum_{\ell\ge1}\E{\1{v_0\text{ has }\ge\ell\text{ children}}\1{\text{index of }\ell\text{th child is }\le k}\left(\text{index of }\ell\text{th child}\right)}\\
\label{eq:E_LW_decomp}
&=\quad\sum_{\ell\ge1}\mu(\ge\ell)\E{\left(\sum_{i=1}^{\ell-1}|\GW_i|+1\right)\1{\sum_{i=1}^{\ell-1}|\GW_i|+1\le k}}.
\end{align}

On one hand this is
\begin{align*}
\le\quad\sum_{\ell\ge1}\mu(\ge\ell)\left((\ell-1)\E{|\GW|\1{|\GW|\le k}}+1\right)\quad\lesssim\quad\left(\sum_{\ell\ge1}\ell\mu(\ge\ell)\right)\cdot\sqrt{k}\quad\lesssim\quad\sqrt{k},
\end{align*}
where we used \Cref{lem:GW_tail_bound} and that $\mu$ has a bounded second moment.

On the other hand \eqref{eq:E_LW_decomp} is
\begin{align*}
\ge\quad\mu(\ge2)\E{|\GW|\1{|\GW|\le k-1}}\quad\asymp\quad\sqrt{k},
\end{align*}
where we used \Cref{lem:GW_tail_bound} and that $\mu$ has mean 1 and variance $\asymp1$, hence $\mu(\ge2)\gtrsim1$.
\end{proof}

\begin{lemma}\label{lem:GW_exp_moment} A critical Galton-Watson tree $\GW$ satisfies
\[1-\E{(1-\alpha)^{|\GW|}}\:\asymp\:\sqrt{\alpha}\qquad\text{as }\alpha\to0,\]
where the implicit constants in $\asymp$ only depend on $\sigma^2$.
\end{lemma}

\begin{proof}
Let $n=\frac{1}{\alpha}$. Then we have
\begin{align*}
\E{(1-\alpha)^{|\GW|}}\quad&\le\quad\pr{|\GW|>n}+\E{(1-\alpha)^{|\GW|}\1{|\GW|\le n}}\\
&\le\quad\pr{|\GW|>n}+\E{\left(1-\frac{\alpha}{2}|\GW|\right)\1{|\GW|\le n}}\quad\\
&=\quad 1-\frac{\alpha}{2}\E{|\GW|\1{|\GW|\le n}},\\
\E{(1-\alpha)^{|\GW|}}\quad&\ge\quad\E{(1-\alpha)^{|\GW|}\1{|\GW|\le n}}\\
&\ge\quad\E{\left(1-\alpha|\GW|\right)\1{|\GW|\le n}}\quad\\
&=\quad 1- \pr{|\GW|>n}-\alpha\E{|\GW|\1{|\GW|\le n}}.
\end{align*}
From \Cref{lem:GW_tail_bound} we know that $\E{|\GW|\1{|\GW|\le n}}\asymp\sqrt{n}\asymp\frac{1}{\sqrt{\alpha}}$, and $\pr{|\GW|>n}\asymp\frac{1}{\sqrt{n}}\asymp\alpha$, which finishes the proof.
\end{proof}

\begin{lemma}\label{lem:LW_LH_exp_moment}
The ladder width $\LW$ and ladder height $\LH$ corresponding to a critical walk with steps $\sim(\mu-1)$ satisfy
\begin{align*}
1-\E{(1-\alpha)^{\LW}}\:&\asymp\:\sqrt{\alpha}\qquad\text{as }\alpha\to0,\\
1-\E{(1-\alpha)^{\LH}}\:&\asymp\:\alpha\qquad\text{as }\alpha\to0,
\end{align*}
where the implicit constants in $\asymp$ only depend on $\sigma^2$.
\end{lemma}

\begin{proof}
Analogous to the proof of \Cref{lem:GW_exp_moment}, using \Cref{lem:LW_LH_tail_bound}.
\end{proof}

\begin{corollary}\label{cor:theta_thetatil_thetadot}
For $\theta$, $\til{\theta}$ and $\dot{\theta}$ as in~\eqref{eq:theta_def} we have
\[1-\theta\quad\asymp\quad1-\til{\theta}\quad\asymp\quad1-\dot{\theta} \quad\asymp\quad\frac{1}{\sqrt{n}},\]
where the implicit constants in $\asymp$ only depend on $\sigma^2$.
\end{corollary}

\begin{proof}
Immediate from Lemmas~\ref{lem:GW_exp_moment} and~\ref{lem:LW_LH_exp_moment}.
\end{proof}

\begin{lemma}\label{lem:theta0_theta_1}
For $\theta_0$ and $\theta_1$ as in~\eqref{eq:theta0_def} we have
\[1-\theta_0\quad\asymp\quad1-\theta_1\quad\asymp\quad\frac{1}{\sqrt{n}},\]
where the implicit constants in $\asymp$ only depend on $\sigma^2$.
\end{lemma}

\begin{proof}
By considering the number of offsprings of the root in a $\GW$ tree, we see that \[\til{\theta}\quad=\quad\E{\lambda^{|\GW|}}\quad=\quad\E{\lambda\:\E{\lambda^{|\GW|}}^{d_0^+}}\quad=\quad\lambda\:\E{\til{\theta}^{d_0^+}},\]
i.e. $\theta_0=\lambda^{-1}\til{\theta}$. Using \Cref{cor:theta_thetatil_thetadot} and that $\lambda=1-\frac1n$, we get that $1-\theta_0\asymp\frac{1}{\sqrt{n}}$.

To get the estimate for $\theta_1$, note that
\begin{align*}
\E{\til{\theta}^{d_1^+}}\quad&=\quad\sum_{k\ge0}\til{\theta}^k\:\pr{d_1^+=k}\quad=\quad\sum_{k\ge0}\sum_{\ell\ge0}\til{\theta}^k\:\mu(k+\ell+1)\quad=\quad\sum_{m\ge1}\mu(m)\sum_{k=0}^{m-1}\til{\theta}^k\quad\\
&=\quad\sum_{m\ge1}\mu(m)\frac{1-\til{\theta}^m}{1-\til{\theta}}\quad=\quad\sum_{m\ge0}\mu(m)\frac{1-\til{\theta}^m}{1-\til{\theta}}\quad=\quad\frac{1-\theta_0}{1-\til{\theta}},
\end{align*}
hence $1-\theta_1=\frac{\theta_0-\til{\theta}}{1-\til{\theta}}=\frac{\lambda^{-1}\til{\theta}(1-\lambda)}{1-\til{\theta}}\asymp\frac{1}{\sqrt{n}}$ as required.
\end{proof}

\begin{lemma}\label{lem:theta_fraction}
We have
\[\frac{1-\theta_1}{1-\til{\theta}}\quad\sim\quad\frac{\sigma^2}{2}\:.\]
\end{lemma}

\begin{proof} A quick calculation shows that $\E{d_1^+}=\frac{\sigma^2}{2}$. Note that we have
\begin{align*}
\frac{1-\theta_1}{1-\til{\theta}}\quad=\quad\frac{1-\E{\til{\theta}^{d_1^+}}}{1-\til{\theta}}\quad=\quad\E{\sum_{\ell=0}^{d_1^+-1}\til{\theta}^{\ell}}\quad\le\quad\E{d_1^+}.
\end{align*}
For the lower bound note that
\[
\E{\sum_{\ell=0}^{d_1^+-1}\til{\theta}^{\ell}}\quad\ge\quad\sum_{\ell=0}^{L}\til{\theta}^{\ell}\pr{d_1^+\ge\ell+1}\quad\ge\quad \til{\theta}^{L}\sum_{\ell=0}^{L}\pr{d_1^+\ge\ell+1}.
\]
We know that $\sum_{\ell=0}^{L}\pr{d_1^+\ge\ell+1}\uparrow\E{d_1^+}$ as $L\to\infty$, and that for any fixed value of $L$, we have $\til{\theta}^L\uparrow1$ as $n\to\infty$. So choosing $L$ sufficiently large and then $n$ sufficiently large in terms of $L$, we can make the right-hand side arbitrarily close to $\E{d_1^+}$. This finishes the proof.
\end{proof}

\begin{lemma}\label{lem:sum_GW_sizes} For any $s>1$ we have
\[\frac{1}{\sqrt{s}}\quad\lesssim\quad\pr{\sum_{i=1}^n|\GW_i|\ge n^2s}\quad\lesssim\quad\frac{\sqrt{\log s}}{\sqrt{s}},\]
where the implicit constants in $\lesssim$ only depend on $\sigma^2$.
\end{lemma}

\begin{proof}
Using \Cref{lem:GW_tail_bound} we can lower bound as
\begin{align*}
\pr{\sum_{i=1}^n|\GW_i|\ge n^2s}\quad\ge\quad\pr{|\GW_i|\ge n^2s\text{ for some }i\in\{1,2,...,n\}}\\
=\quad1-\left(1-\pr{|\GW|\ge n^2s}\right)^n\quad\ge\quad 1-\left(1-\frac{c_1}{n\sqrt{s}}\right)^n\quad\gtrsim\quad\frac{1}{\sqrt{s}}\:.
\end{align*}
For the upper bound let $m=\frac{n^2s}{\log s}$, and let $\xi_m\sim\Geomnonneg{\frac{1}{m}}$. Let $(\GW_i)$ be an infinite sequence of independent $\GW$ trees and let us index their vertices by the nonnegative integers, going around anticlockwise on each tree in order. Then the number $\zeta$ of trees with all vertices having index $\le\xi_m$ is distributed as $\Geomnonneg{1-\til{\theta}}$, and from \Cref{lem:GW_exp_moment} we know that $1-\til{\theta}\asymp\frac{1}{\sqrt{m}}$. Using these we get
\begin{align*}
\pr{\sum_{i=1}^n|\GW_i|\ge n^2s}\quad\le\quad\pr{\xi_m\ge n^2s}+\pr{\zeta<n}\quad\lesssim\quad e^{-\frac{n^2s}{m}}+\frac{n}{\sqrt{m}}\quad\asymp\frac{\sqrt{\log s}}{\sqrt{s}}\:.
\end{align*}
This finishes the proof.
\end{proof}

\section{Estimates regarding SRWs}\label{app:SRW}

\begin{lemma}[\cite{RW_modern_intro}, Proposition 2.4.5] \label{lem:hitting_time_tail_bounds}
Let $X$ be a SRW on $\Z^d$, starting from 0, let $R>0$, let $\tau_R:=\inf\{n\ge0:\:||X_n||\ge R\}$ and let $s\ge1$. Then we have
\[
\pr{\tau_R>sR^2}\quad\lesssim\quad e^{-cs}
\]
where $c$ is a positive constant.

Also, for $\alpha\lesssim1$ we have
\[
\pr{\tau_R<\alpha R^2}\quad\lesssim\quad e^{-\frac{\til{c}}{\alpha}},
\]
where $\til{c}$ is a positive constant.
\end{lemma}

\begin{lemma}\label{lem:g_alpha_x_bound}
If $d>2$, $\alpha\ll1$, and $x\in\Z^d$ satisfies $||x||\ge\frac{c_0}{\sqrt{\alpha}}$ for some positive constant $c_0$, then
\begin{align*}
g_{\alpha}(x)\quad\lesssim\quad e^{-c||x||\sqrt{\alpha}}g(x),
\end{align*}
where $c$ is a positive constant.

Also, if $||x||\le\frac{C_0}{\sqrt{\alpha}}$ for some positive constant $C_0$, then
\begin{align*}
g(x)-g_{\alpha}(x)\quad\lesssim\quad\left(x\sqrt{\alpha}\right)^{\frac{2d(d-2)}{d^2+2d-4}}g(x).
\end{align*}
\end{lemma}

\begin{proof}
For the first bound note that $g_{\alpha}(x)\lesssim\pr{\tau_{\frac12||x||}<\Geomnonneg{\alpha}}g(x)$, and\\ $\pr{\tau_{\frac12||x||}<\Geomnonneg{\alpha}}\le\pr{\tau_{\frac12||x||}<\frac{||x||}{\sqrt{\alpha}}}+\pr{\Geomnonneg{\alpha}>\frac{||x||}{\sqrt{\alpha}}}$. The last two probabilities are both $\lesssim e^{-c||x||\sqrt{\alpha}}$.

For the second bound, start by noting that for any $t$ and $R$ with $R\ge2||x||$,  we have
\begin{align*}
&\E{\#\left(\text{visits at }x\text{ after time }(t-1)\right)}\quad=\quad\E{g(X_t,x)}\\
&\lesssim\quad\sum_{y\in B\left(x,\frac12||x||\right)}\pr{X_t=y}g(x,y)+
\sum_{y\in B(0,R)\setminus B\left(x,\frac12||x||\right)}\pr{X_t=y}g(x,y)\\
&\hspace{9cm}+\sum_{y\in B(0,R)^c}\pr{X_t=y}g(x,y)\\
&\lesssim\quad \sum_{y\in B\left(x,\frac12||x||\right)} \frac{1}{t^{\frac{d}{2}}}\frac{1}{||x-y||^{d-2}}\quad+\quad
\frac{R^d}{t^{\frac{d}{2}}}\frac{1}{||x||^{d-2}}\quad+\quad
\frac{1}{R^{d-2}}
\quad\\
&\hspace{8cm}\lesssim\quad g(x)\left(\frac{||x||^d}{t^{\frac{d}{2}}}\:+\:\frac{R^{d}}{t^{\frac{d}{2}}}\:+\:\frac{||x||^{d-2}}{R^{d-2}}\right).
\end{align*}
Then for $||x||\le\frac{C_0}{\sqrt{\alpha}}$ we have
\begin{align*}
&g(x)-g_{\alpha}(x)\quad=\quad\E{\#\left(\text{visits at }x\text{ after time }\Geomnonneg{\alpha}\right)}\\
&\qquad\le\quad \pr{\Geomnonneg{\alpha}\le t}g(x) \quad+\quad\E{\#\left(\text{visits at }x\text{ after time }(t-1)\right)}\\
&\qquad\lesssim\quad g(x)\left(t\alpha\:+\:\frac{||x||^d}{t^{\frac{d}{2}}}\:+\:\frac{R^{d}}{t^{\frac{d}{2}}}\:+\:\frac{||x||^{d-2}}{R^{d-2}}\right).
\end{align*}
Choosing $R$ and $t$ such that $t\alpha\asymp\frac{R^{d}}{t^{\frac{d}{2}}}\asymp\frac{||x||^{d-2}}{R^{d-2}}$ and $R\ge2||x||$, we get that this is $\lesssim\left(x\sqrt{\alpha}\right)^{\frac{2d(d-2)}{d^2+2d-4}}g(x)$ as required.
\end{proof}

\begin{lemma}\label{lem:galpha_weighted}
Let $\alpha\ll1$ and $x\in\Z^d$. Then we have
\begin{align*}
\sum_{k\ge1}k(1-\alpha)^k\pr{X_k=x}\quad\lesssim\quad\sum_{k\ge1}k\pr{X_k=x}\quad\asymp\quad (g\star g)(x).
\end{align*}
In case $d>2$ and $||x||\ge\frac{c_0}{\sqrt{\alpha}}$ for a positive constant $c_0$, we also have
\begin{align*}
\sum_{k\ge1}k(1-\alpha)^k\pr{X_k=x}\quad\lesssim\quad e^{-c_1||x||\sqrt{\alpha}}(g\star g)(x).
\end{align*}
Also, for $d>2$ and $x,y\in\Z^d$, we have
\begin{align*}
\sum_{k\ge1}k(1-\alpha)^k\pr{X_k=x}\pr{X_k=y}\quad&\lesssim\quad\sum_{k\ge1}k\pr{X_k=k}\pr{X_k=y}\quad\\
&\asymp\quad \frac{1}{||x||^{2d-4}+||y||^{2d-4}},
\end{align*}
and in case $||x||+||y||\ge\frac{c_0}{\sqrt{\alpha}}$ for a positive constant $c_0$, we also have
\begin{align*}
\sum_{k\ge1}k(1-\alpha)^k\pr{X_k=x}\pr{X_k=y}\quad\lesssim\quad e^{-c_2(||x||+||y||)\sqrt{\alpha}}\frac{1}{||x||^{2d-4}+||y||^{2d-4}}.
\end{align*}
The positive constants $c_1$ and $c_2$ only depend on $\sigma^2$, while the implicit constants in the $\lesssim$'s and $\asymp$'s only depend on $\sigma^2$ and $c_0$.
\end{lemma}

\begin{proof}
The $\lesssim$ and $\asymp$ in the first line are immediate. To prove the second $\lesssim$, start by writing
\begin{align*}
&\sum_{k\ge1}k(1-\alpha)^k\pr{X_k=x}\quad=\quad\sum_{k\ge1}\sum_{\ell=1}^{k}(1-\alpha)^k\pr{X_k=x}\quad\\
&=\quad\sum_{\ell\ge1}\sum_{k\ge\ell}(1-\alpha)^k\pr{X_k=x}\quad=\quad\sum_{\ell\ge1}(1-\alpha)^{\ell}\E{g_{\alpha}(x-X_{\ell})}\quad\\
&=\quad\sum_{y}\sum_{\ell\ge1}(1-\alpha)^{\ell}\pr{X_{\ell}=y}g_{\alpha}(x-y)\quad=\quad\sum_{y}g_{\alpha}(y)g_{\alpha}(x-y).
\end{align*}
Note that for any $y$ we have $||y||\ge\frac12||x||$ or $||x-y||\ge\frac12||x||$, so in case $||x||\ge\frac{c_0}{\sqrt{\alpha}}$, by \Cref{lem:g_alpha_x_bound} we have $\sum_{y}g_{\alpha}(y)g_{\alpha}(x-y)\lesssim\sum_{y}e^{-c_1||x||}g(y)g(x-y)$ as required.

Note that $\pr{X_k=x}\pr{X_k=y}=\pr{(X_k,Y_k)=(x,y)}$ where $X$ and $Y$ are two independents SRWs on $\Z^d$. Then $(X,Y)$ is a SRW on a copy of $\Z^{2d}$ induced by the pairs of points in $\Z^d$ where the sum of their coordinates is even. We can finish the proof by using the results from the first part of the lemma for $(X,Y)$.
\end{proof}

\section{Some further auxiliary results}\label{app:aux}

\begin{lemma}[\cite{RW_modern_intro}, Lemma 6.3.7]\label{lem:Harnack}
Given a positive radius $r$, points $x\in\partial B(0,r)$, $y\in\partial B(0,4r)$ in $\Z^d$, and random walk $X$ on $\Z^d$ from 0, we have \[\pr{X_{\tau_r}=x,X_{\tau_{4r}}=y}\quad\asymp\quad\pr{X_{\tau_r}=x}\pr{X_{\tau_{4r}}=y},\]
where the implicit constants in $\asymp$ do not depend on $r$.
\end{lemma}

\begin{lemma}\label{lem:prob_returning_to_partialBk}
Let $d>2$ and $k\ge1$. Let $X$ be a SRW on $\Z^d$ and let $\tau_k^+=\inf\left\{t\ge1:\:X_t\in\partial B_k\right\}$. Then for any $x\in\partial B_k$ we have
\begin{equation*}
\prstart{\tau_k^+<\infty}{x}\quad=\quad1-\Theta\left(\frac{1}{k}\right),
\end{equation*}
where the implicit constants in $\Theta$ do not depend on $k$ or $x$.
\end{lemma}

The proof of this lemma follows from using the optional stopping theorem for $\estart{g\left(X_{\tau_k^+\wedge t}\right)}{x}$, using the asymptotics for $g$, and taking $t\to\infty$.

\bibliographystyle{plain}
\bibliography{bib_BRW.bib}

\end{document}